\documentclass{amsart}
\usepackage{amssymb,amsmath,amscd,amsthm}
\usepackage{amsfonts,epsfig,latexsym,graphicx,amssymb}
\usepackage{color}
\usepackage{enumerate}
\usepackage{todonotes}

\usepackage{hyperref}

\date{\today}

\newcommand{\Z}{{\mathbb Z}}
\newcommand{\R}{{\mathbb R}}

\newcommand{\N}{{\mathbb N}}

\newcommand{\Diff}{\mathrm{Diff}}
\newcommand{\Jac}{\mathop{\mathrm{Jac}}}
\newcommand{\Id}{\mathop{\mathrm{Id}}}
\newcommand{\bS}{\mathbb{S}}

\newcommand{\Bq}{Q}
\newcommand{\Bi}{B'}

\newcommand{\qh}{\mathbf{h}}
\newcommand{\Gr}{\mathrm{Gr}}

\newcommand{\mE}{\mathcal{E}}
\newcommand{\mI}{\mathcal{I}}
\newcommand{\mF}{\mathcal{F}}

\newcommand{\Leb}{{\mathrm{Leb}}}

\newcommand{\SL}{{\mathrm{SL}}}

\newcommand{\Homeo}{\mathrm{Homeo}}

\newcommand{\const}{\mathop{\mathrm{const}}}

\newcommand{\bv}{\bar{v}}

\newcommand{\Dirac}{\delta}

\newcommand{\hB}{\hat B}
\newcommand{\tB}{\widetilde{B}}

\newcommand{\wT}{\widetilde{T}}

\newcommand{\mN}{\mathcal{N}}

\newcommand{\bK}{\mathbf{K}}
\newcommand{\mgr}{\mu}
\newcommand{\msp}{\nu}
\newcommand{\bmsp}{\overline{\msp}}

\newcommand{\RP}{\mathbb{RP}}

\newcommand{\Ind}{\mathbf{1}}

\newcommand{\rxi}{\eta} 

\newcounter{mit}

\newcounter{Bit}

\newtheorem{theorem}{Theorem}[section]
\newtheorem*{theorem*}{Theorem}
\newtheorem{lemma}[theorem]{Lemma}
\newtheorem{prop}[theorem]{Proposition}
\newtheorem{coro}[theorem]{Corollary}

\newtheorem{example}[theorem]{Example}
\theoremstyle{definition}
\newtheorem{remark}[theorem]{Remark}

\newtheorem{defi}[theorem]{Definition}

\def\dist{{\rm dist}}

\newcommand{\bi}{{\bf i}}

\def\N{{\mathbb N}}

\newcommand{\E}{{\mathbb E}\,}
\newcommand{\Prob}{{\mathbb P}\,}
\def\P{\Prob}
\newcommand{\eps}{{\varepsilon}}
\begin{document}

\title[Non-stationary  Furstenberg Theorem]{Non-stationary version of Furstenberg Theorem on random matrix products}

\author[A.\ Gorodetski]{Anton Gorodetski}

\address{Department of Mathematics, University of California, Irvine, CA~92697, USA}

\email{asgor@uci.edu}

\thanks{A.\ G.\ was supported in part by Simons Fellowship and NSF grant DMS--1855541.}

\author[V. Kleptsyn]{Victor Kleptsyn}

\address{CNRS, Institute of Mathematical Research of Rennes, IRMAR, UMR 6625 du CNRS}

\email{victor.kleptsyn@univ-rennes1.fr}

\thanks{V.K. was supported in part by ANR Gromeov (ANR-19-CE40-0007) and by Centre Henri Lebesgue (ANR-11-LABX-0020-01) }

\begin{abstract}
We prove a non-stationary analog of the Furstenberg Theorem on random matrix products (that can be considered as a matrix version of the law of large numbers). Namely, we prove that under a suitable genericity conditions the sequence of norms of random products of independent but not necessarily identically distributed $\SL(d, \mathbb{R})$ matrices grow exponentially fast, and there exists a non-random sequence that almost surely describes asymptotical behaviour of the norms of random products.
\end{abstract}

\maketitle

\section{Introduction}

The asymptotic behavior of sums of i.i.d. random variables is very well studied in the classical probability theory. Analogous questions on random products of matrix-valued i.i.d. random variables were initially formulated in the simplest case of $2\times 2$ matrices with positive entries by Bellman \cite{Bel}. Later these questions attracted lots of attention due to the results by Furstenberg-Kesten \cite{FK} who showed that exponential rate of growth of the norms of the random products (nowadays called Lyapunov exponent) is well defined almost surely, and Furstenberg \cite{Fur1, Fur2}, where it was shown that under some non-degeneracy conditions Lyapunov exponent must be positive. Since then enormous amount of literature on the subject appeared, e.g. see \cite{Ber, Fur3, FKif, GM, GR, KS, Kif, KifS, L, R, SVW, Vi}. Applications of random matrix products appear in a natural way in smooth dynamical systems \cite{V, W1, W2}, spectral theory and mathematical physics \cite{CKM, D, S}, geometric measure theory \cite{HS, PT, Sh}, and other fields. Far reaching generalizations in terms of random walks on groups were developed, see \cite{BQ}, \cite{Fu}, and references therein. Nonlinear one-dimensional analogues of Furstenberg Theorem were also obtained in~\cite{A,DKN1,KN,M}. Another series of generalizations (in terms of positivity of Lyapunov exponents for a generic linear cocycle) was derived in the dynamical systems community, e.g. see \cite{ASV}, \cite{BGV}, \cite{Bo}, \cite{BoV1}, \cite{BoV2}, \cite{BoV3}, \cite{V1}, and the monograph \cite{V}.

  The most famous result is the following Furstenberg Theorem, that we recall here in its classical form:

\begin{theorem*}[Furstenberg~{\cite[Theorem~8.6]{Fur1}}]
Let $\{X_k, k\ge 1\}$ be independent and identically distributed random variables, taking values in $\SL(d, \mathbb{R})$, the $d\times d$ matrices with determinant one, let $G_X$ be the smallest closed subgroup of $\SL(d, \mathbb{R})$ containing the support of the distribution of $X_1$, and assume that
$$
E[\log\|X_1\|]<\infty.
$$
Also, assume that $G_X$ is not compact, and there exists no $G_X$-invariant finite union of proper subspaces of~$\R^d$. Then there exists a positive constant $\lambda_F$ such that with probability one
$$
\lim_{n\to \infty}\frac{1}{n}\log\|X_n\ldots X_2X_1\|=\lambda_F>0.
$$
\end{theorem*}

In this paper we generalize Furstenberg Theorem  to the case when the random variables $\{X_k, k\ge 1\}$ do not have to be identically distributed. Let us say a couple of words about a motivation for our results before providing the formal statements.

First of all, the study of asymptotic behavior of sums $S_n=x_n+\ldots +x_1$ of independent real valued random variables, in particular, the Law of Large Numbers, is a foundational statement in classical probability theory. As we mentioned above, many of the limit theorems were proven for that non-commutative case. In particular, Furstenberg Theorem can be considered as a non-commutative version of the Law of Large Numbers.  But in all of the obtained results on random matrix products the matrices were assumed to be identically distributed, while most (if not all) of the results on sums $S_n=x_n+\ldots +x_1$  of independent real valued random variables do not actually require $x_1, \ldots, x_n$ to be identically distributed. It is a very natural task to close this gap.  

As an important application, we mention the non-stationary version of so called Anderson Model. Classical Anderson Model on a one-dimensional lattice is given by a discrete Schr\"odinger operator on $l^2(\Z)$ with random potential:
$$
(Hu)(n)=u(n-1)+u(n+1)+V(n)u(n),
$$
where $\{V(n)\}$ is an iid sequence of random variables. It is known that under suitable assumptions this operator almost surely has pure point spectrum, with exponentially decreasing eigenfunctions (this property is usually referred to as {\it Anderson Localization}). Other properties of this operators (such as dynamical localization, properties of the integrated density of states etc.) were also studied in details, see \cite{AW, D, DKKKR} for some recent surveys. But from the physical point of view it is very natural to consider the case when the potential $\{V(n)\}$ is given by independent but not identically distributed random variables. Indeed, suppose we try to study the transport properties of an electron in the random media with some fixed non-random background. In this case it is natural to consider potential $V(n)=V_{random}(n)+V_{background}(n)$, where $\{V_{background}(n)\}$ is a fixed bounded sequence, and $\{V_{random}(n)\}$ is a sequence of iid random variables. Since the spectral properties of the 1D discrete Schr\"odinger operator are closely related to the properties of the products of the corresponding transfer matrices, this setting leads to the question about properties of a random matrix products in the non-stationary case. In particular, the results of this paper allow to prove spectral and dynamical localization in non-stationary Anderson Model, as we plan to show in our next paper \cite{GK}.

Up to now there were literally no results on random matrix products in non-stationary case, since there were no suitable techniques available. Indeed, in most cases the proofs in the stationary case are based on existence of a stationary measure, which restricts all the existing techniques either to the case of identically distributed matrices (or, at least, with the distributions given by some stationary process, as in \cite{Kif2}), or to the context of Oseledets Theorem~\cite{O} (see also~\cite{R1}), with some exceptions that are usually focused on specific models, with the proofs heavily based on the special features of the model. Our first main result, Theorem~\ref{t.L},  describes the growth of the norms of random matrix products in the case of independent but not identically distributed matrices. As an intermediate step, we provide a general result, Theorem \ref{p.max.full}, that we called Atom Dissolving Theorem, that can be used in non-stationary context as a tool to replace the statements on regularity of stationary measures, but is also of independent interest, see Section \ref{ss.ad}.

Let us now provide the formal statements of our main results.

\subsection{Non-stationary version of Furstenberg theorem}


Let $\bK$ be a compact set of probability measures on $\SL(d, \mathbb{R})$; as a particular case, one can consider $\bK=\{\mgr_i\}_{i=1, \ldots, k}$ being a finite set.
For any $A\in \SL(d, \mathbb{R})$ we will denote by $f_A:\RP^{d-1}\to \RP^{d-1}$ the induced projective transformation.

For a given sequence $(\mgr_i)_{i\in \mathbb{N}}$, $\mgr_i\in \bK$,
we let $A_i\in \SL(d, \mathbb{R})$ be chosen randomly with respect to distribution $\mgr_i$, set
\[
T_n=A_nA_{n-1}\ldots A_1,
\]
and denote
\begin{equation}\label{eq:L-def}
L_n=\E\log\|T_n\|,
\end{equation}
where the expectation is taken over the distribution $\mgr_1\times \mgr_2\times\ldots\times \mgr_n$. This expectation exists once the log-moment of the norm $\E \log \|A\|$ is finite for all $\mgr\in\bK$, and we will be always imposing (at least) this assumption.

Our first main result is the following theorem:

\begin{theorem}\label{t.L}
Assume that the following hold:
\begin{itemize}
\item
\textbf{(finite moment condition)} There exists $\gamma>0$, $C$ such that
\begin{equation}\label{eq:finite-moment}
\forall \mgr\in \bK\quad  \int_{\SL(d,\R)} \|A\|^{\gamma} d \mgr(A) < C
\end{equation}
\item
\textbf{(measures condition)} For any $\mgr\in \bK$ there are no Borel probability measures $\msp_1$, $\msp_2$ on $\RP^{d-1}$ such that $(f_A)_*\msp_1=\msp_2$ for $\mgr$-almost every $A\in \SL(d, \mathbb{R})$
\item
\textbf{(spaces condition)} For any $\mgr\in \bK$ there are no two finite unions $U$, $U'$ of proper subspaces of~$\R^{d}$ such that $A(U)=U'$ for $\mgr$-almost every $A\in \SL(d, \mathbb{R})$.
\end{itemize}
Then the sequence $L_n=\E \log \| T_n \|$ grows at least linearly, i.e. then there exists $\lambda>0$ such that for any $n$ and any $\mgr_1,\dots,\mgr_n\in \bK$ we have
\[
L_n \ge n \lambda,
\]
and it predicts the growth of the norm of the random products in the following sense: almost surely, we have
\[
\lim_{n\to \infty}\frac{1}{n}\left(\log\|T_n\|-L_n\right)=0.
\]
\end{theorem}

The particular case of $2\times 2$ matrices, and the existence of an exponentially contracted random vector in that case, are important in the setting of 1D Anderson Localization (see~\cite{D,His}). For this case, Theorem~\ref{t.L} has the following addendum:
\begin{prop}[Contracted direction]\label{t:1.5}
Let $d=2$, and assume that the finite moment and measures conditions
of Theorem~\ref{t.L} hold. Then almost surely there exists a unit vector $\bar v\in \mathbb{R}^2$ such that $|T_n\bar v|\to 0$ as $n\to \infty$. Moreover,
$$
\lim_{n\to \infty}\frac{1}{n}\left(\log|T_n\bar v|+L_n\right)=0
$$
\end{prop}

We will prove Theorem~\ref{t.L} (and thus Proposition~\ref{t:1.5}) by actually establishing a stronger conclusion, the Large Deviations Estimates Theorem:
\begin{theorem}[Large Deviations for Nonstationary Products]\label{t.2}
Under the assumptions of Theorem~\ref{t.L},
for any $\varepsilon>0$ there exists $\delta>0$ such that for all sufficiently large $n\in \mathbb{N}$ we have
$$
\mathbb{P}\left\{\left|\log\|T_n\|-L_n\right|>\varepsilon n\right\}<e^{-\delta n},
$$
where $\mathbb{P}=\mgr_{1}\times \mgr_{2}\times\ldots\times \mgr_{n}$. Moreover, the same estimate holds for the lengths of random images of any given initial unit vector $v_0$:
$$
\forall v_0\in \R^d, \, |v_0|=1 \quad  \mathbb{P}\left\{\left|\log\|T_n v_0\|-L_n\right|>\varepsilon n\right\}<e^{-\delta n}.
$$
\end{theorem}

\begin{remark}
The lower bound on $n$ in Theorem \ref{t.2} depends on $\varepsilon>0$ and the compact set of distributions $\mathbf{K}$, but can be chosen uniformly over all the specific choices of the sequences of measures $\mu_1, \mu_2, \ldots$ from $\mathbf{K}$ and, in the second part, the unit vector $v_0$.
\end{remark}

\subsection{Exponential growth of the norms}

The sequence $\{L_n\}$ in Theorem \ref{t.L} must grow at least linearly. This statement by itself holds under weaker assumption than those in Theorem  \ref{t.L}, and so we formulate the statement on growth of the random matrix products separately:
\begin{theorem}\label{t.1}
Assume that the following hold:
\begin{itemize}
\item \textbf{(log-moment condition)} For any $\mgr\in \bK$ one has
\[
\int_{\SL(d,\R)} \log \|A\| d \mgr(A) < \infty
\]
\item \textbf{(measures condition)} For any $\mgr\in \bK$ there are no Borel probability measures $\msp_1$, $\msp_2$ on $\RP^{d-1}$ such that $(f_A)_*\msp_1=\msp_2$ for $\mgr$-almost every $A\in \SL(d, \mathbb{R})$
\end{itemize}
Then there exists $\lambda>0$ such that for any $n$ and any $\mgr_1,\dots,\mgr_n\in \bK$ we have
\[
L_n \ge n \lambda.
\]
In particular, for any fixed sequence $(\mgr_i)_{i\in \mathbb{N}}\in \bK^{\mathbb{N}}$ we have
\[
\liminf_{n\to \infty}\frac{1}{n} L_n \ge \lambda >0.
\]
\end{theorem}

Questions regarding exponential growth of nonstationary random matrix products were discussed and popularized by I.~Goldsheid for a long time. In his recent paper~\cite{G4}, it was shown that under the same ``measure condition'' as in Theorem~\ref{t.1},
there exists a positive $\lambda$ such that almost surely
\[
\liminf_{n\to\infty} \frac{1}{n} \log \|T_n\| \ge \lambda>0.
\]
The proof was obtained by completely different methods.

Our methods allow to consider Theorem~\ref{t.1} as a particular case of a more general result. Namely, consider any closed manifold $M$ and the set of its $C^1$-diffeomorphisms~$\Diff^1(M)$. Assume that $M$ is equipped with a Riemannian metric, so that one can consider the Lebesgue measure $\Leb_M$ on $M$ and the corresponding Jacobian
\[
\Jac(f)|_x = \left |\det df|_x  \right| = \left. \frac{d\Leb_M}{d f_* \Leb_M}\right|_x.
\]
We will measure the maximum volume contraction rate of a diffeomorphism by the following quantity:
\[
\mN(f):=\max_{x\in M} \Jac(f)|_x^{-1} = \max_{x\in M} \left. \frac{df_* \Leb_M}{d \Leb_M}\right|_x.
\]
Let $\bK_M$ be a compact subset of the space of probability measures on $\Diff^1(M)$ (equipped with the weak-$*$ convergence topology). We then have the following theorem, providing a lower estimate for the (averaged) growth of the maximum volume contraction speed:
\begin{theorem}\label{t:M}
Let $M$, $\bK_M$ satisfy the following assumptions:
\begin{itemize}
\item \textbf{(log-moment condition)}
For any $\mgr\in \bK_M$ one has
\[
\int_{\Diff^1(M)} \log \mN(f) d \mgr(f) < \infty
\]
\item \textbf{(measures condition)} For any $\mgr\in \bK_M$ there are no Borel probability measures $\msp_1$, $\msp_2$ on $M$ such that $f_*\msp_1=\msp_2$ for $\mgr$-almost every $f\in \Diff^1(M)$.
\end{itemize}
Then there exists $\qh>0$ such that for any $n$ and any $\mgr_1,\dots,\mgr_n\in\bK_M$ we have
\[
\E \log \mN(F_n) \ge n \qh,
\]
where $F_n=f_n\circ \dots \circ f_1$, and every $f_i$ is chosen independently with respect to the corresponding measure~$\mgr_i$, so that the expectation is taken over the distribution $\mgr_1\times \mgr_2\times \ldots \times \mgr_n$.

In particular, for any given sequence $(\mgr_i)_{i\in\N}$, $\mgr_i\in\bK_M$ of measures on $\Diff^1(M)$ we have
\[
\liminf_{n\to\infty} \frac{1}{n} \E \log \mN(F_n) \ge \qh>0,
\]
where $F_n=f_n\circ \dots \circ f_1$, and the expectation is taken with respect to the infinite product measure~$\prod_i \mgr_i$.
\end{theorem}

\begin{remark}
Theorem~\ref{t:M} could be considered as a generalization of the famous Baxendale Theorem~\cite{Bax}, claiming (in the stationary case) the existence of an ergodic measure with a negative volume Lyapunov exponent in the case of absence of a common invariant measure.
\end{remark}

\begin{remark}
 The conclusions of Theorems~\ref{t.1} and~\ref{t.L} in the case of products of i.i.d. random matrices correspond to the classical Furstenberg Theorem. Proposition~\ref{t:1.5} can be considered as a non-stationary analog of Proposition~II.3.3 and Corollary~IV.1.7 from~\cite{BL}, or of results from~\cite{R1}.%
\end{remark}
\begin{remark}\label{r:groups}
One can replace the assumptions of Theorem~\ref{t.1} by a more general one. Namely, instead of the measures condition, it is enough to assume that there exists $k\in \N$ such that the assumptions of absence of measures and finite unions of subspaces with deterministic image hold for some $k$ for the $k$-fold convolutions
\[
\bK^{*k} =\{ \mgr_1 * \dots * \mgr_k \mid \mgr_1 , \dots , \mgr_k \in \bK\},
\]
where $\mgr_1* \dots * \mgr_k$ is the law of the product $A_1\dots A_k$, where each $A_i$ is chosen independently w.r.t.~$\mgr_i$.

This version can be immediately reduced to the initial one (it suffices to group the matrices into finite products of length $k$). Nevertheless, it can be useful. For example, it allows to cover the case of the distributions $\mgr$ supported on just two points (that is needed to treat Anderson--Bernoulli type models).

Indeed, the measures condition never holds for a distribution supported just on two matrices $A, B\in \SL(d, \mathbb{R})$. To see that, one can take $\msp_1$ to be an invariant measure of the map $f_{B^{-1}A}:\RP^{d-1}\to \RP^{d-1}$, and notice that
 \[
 \msp_2=(f_B)_*\msp_1=(f_B)_*(f_{B^{-1}A})_*\msp_1=(f_A)_*\msp_1.
 \]
But $\bK^{*k}$ with $k>1$ generally will have lager support, and using this more general condition rectifies the situation.
\end{remark}

\begin{remark}\label{rq:notes}
It is interesting to compare the assumptions of Theorem~\ref{t.1} to other non-degeneracy assumptions that were used by different authors in the stationary setting.
\begin{enumerate}
\item The original Furstenberg assumption (support of the distribution is not contained in a compact subgroup of $\SL(d, \mathbb{R})$, and  there is no finite union of proper subspaces of~$\R^d$ that would be invariant under almost every linear map) in the case $d=2$ is equivalent to the stationary analog of the measures condition (there is no measure that would be preserved by almost every transformation). Indeed, existence of finite union of lines invariant under almost every map is equivalent to existence of an atomic measure on $\RP^1$  invariant under almost every map, and the support of the distribution is inside of a compact subgroup of $\SL(2, \mathbb{R})$  if and only if there  exists a non-atomic probability  measure on $\RP^1$  invariant under almost every map, see \cite[Lemma 3.6]{AB1}.
\item In the case $d>2$ the measures condition is weaker than the Furstenberg assumption. Notice that the conclusion of Theorem \ref{t.1} in the stationary case is also weaker than the conclusion of the classical Furstenberg Theorem. Moreover, Theorem~\ref{t.L} does not hold for arbitrary $d$ under the assumptions of Theorem~\ref{t.1} only: we provide the corresponding example in Appendix~\ref{a.1}, see Example~\ref{ex:R-4}. \item Other non-degeneracy assumptions in the stationary case were also used. For example, in~\cite{GM} the non-degeneracy assumption is given in terms of algebraic richness of the support of the distribution (to get simple Lyapunov exponents), and in \cite{GR} - in terms of strong irreducibility and  ``contracting condition'' (to get simple largest Lyapunov exponent). In both cases, the measures condition follows from those sets of assumptions.
\end{enumerate}
\end{remark}

\subsection{Atom Dissolving and Subspaces Avoidance}\label{ss.ad}

One of the key steps of the proof of Theorem~\ref{t.2} is to show that under a long random composition the probability that a given initial vector is sent to a given (hyper)plane tends to zero as the length of the composition increases. In particular, the probability that for the projectivized dynamics a point of $\RP^{d-1}$ is sent into a given point should converge to zero: the maximal weight of an atom should decrease. Though such statements are not difficult to show in the stationary setting (due to the existence of a stationary measure), they turn out to be more difficult in a non-stationary setting due to the lack of tools.

Actually, for a general non-stationary case (unavoidably, under ``measures condition'') we show that the maximal weight of an atom decreases \emph{exponentially}. We believe that this ``atoms dissolving'' statement, as well as the one for the projective maps, is of independent interest.

\begin{defi}\label{d.defMax}
Denote by $\mathfrak{Max}(\msp)$ the weight of a maximal atom of a probability measure $\msp$. In particular, if $\msp$ has no atoms, then $\mathfrak{Max}(\msp)=0$.
\end{defi}

We will also be using the following notation:
\begin{defi}
Let a group $G$ be acting on a space $X$ (we will need the cases $G=\SL(d,\R)$, $X=\RP^{d-1}$ and $G=\Diff^1(M)$, $X=M$).  For two probability measures $\mgr,\mgr'$ on $G$, let $\mgr*\mgr'$ be the law of $f g$, where $f$ and $g$ are chosen independently w.r.t. $\mgr$ and $\mgr'$ respectively. Also, for a measure $\mgr$ on $G$ and a measure $\msp$ on $X$, we let $\mgr*\msp = \E_\mgr f_* \nu$ be the law of $f(x)$, where $f\in G$ and $x\in X$ are chosen independently w.r.t. $\mgr$ and $\msp$ respectively.
\end{defi}

\begin{defi}
Let $X$ be a metric compact. For a measure $\mgr$ on the space of homeomorphisms $\Homeo(X)$, we say that there is
\begin{itemize}
\item \emph{no finite set with a deterministic image}, if there are no two finite sets $F,F'\subset X$ such that $f(F)=F'$ for $\mgr$-a.e. $f\in \Homeo(X)$;
\item \emph{no measure with a deterministic image}, if there are no two probability measures $\msp,\msp'$ on $X$ such that $f_*\msp=\msp'$ for $\mgr$-a.e. $f\in \Homeo(X)$.
\end{itemize}
\end{defi}

The first of the above mentioned statements is actually a general statement for non-stationary dynamics, ensuring the ``dissolving of atoms'': decrease of a probability of a given point being sent to any particular point.

\begin{theorem}[Atoms Dissolving]\label{p.max.full}
Let $\bK_X$ be a compact set of probability measures on $\Homeo(X)$.
\begin{itemize}
\item
Assume that for any $\mgr\in \bK_X$ there is no finite set with a deterministic image. Then for any $\eps>0$ there exists $n$ such that for any probability measure $\msp$ on $X$ and any sequence $\mgr_1, \ldots, \mgr_n\in \bK_X$ we have
$$
\mathfrak{Max}\left(\mgr_{n}* \dots* \mgr_{1}*\msp\right) <\eps.
$$
In particular, for any probability measure $\msp$ on $X$ and any sequence $\mgr_1, \mgr_2, \ldots\in \bK_X$ we have
$$
\lim_{n\to \infty}\mathfrak{Max}\left(\mgr_{n}* \dots* \mgr_{1}*\msp\right)=0.
$$
\item If, moreover, for any $\mgr\in \bK_X$ there is no measure with a deterministic image, then the convergence is exponential and uniform over all sequences $\mgr_1, \mgr_2, \ldots$ from $K^{\mathbb{N}}$ and all probability measures $\msp$. That is, there exists $\lambda<1$ such that for any $n$, any $\msp$ and any $\mgr_1,\mgr_2,\dots\in\bK_X$
\[
\mathfrak{Max}\left(\mgr_{n}* \dots* \mgr_{1}*\msp\right) < \lambda^n.
\]
\end{itemize}
\end{theorem}

This statement alone does not suffice for the proof of Theorem~\ref{t.2}, as we have to control the probability that a vector is sent into a (hyper)plane. Hence, we will need its strengthened version for a particular case of projective dynamics.

For any $m=1,\dots, d-1$, let $\Gr(m,d)$ be the set of $m$-dimensional subspaces of~$\R^d$. Also, for $L\in \Gr(m,d)$ we denote by $[L]\subset \RP^{d-1}$ the corresponding $(m-1)$-dimensional projective subspace.

\begin{prop}[Subspaces Avoidance]\label{p:proj-dissolving}
Assume that the assumptions of Theorem~\ref{t.1} hold. Moreover, assume that for some $m$ with $1\le m \le d-1,$ for any $j=1,\dots, m$ there are no measure $\mgr\in \bK$ and finite unions $U,U'$ of $j$-dimensional subspaces of $\R^d$ such that
\[
A(U)=U' \quad \text{for $\mgr$-a.e. $A\in \SL(d,\R)$}.
\]
Then for any $\eps>0$ there is a number of iterations $k_m(\eps)$ such that
\begin{multline*}
\forall k\ge k_m(\eps) \quad \forall x_0\in \RP^{d-1} \quad \forall L \in \Gr(m,d) \\
\forall \mgr_1,\dots, \mgr _k \in \bK \quad \Prob_{\mgr_1,\dots,\mgr_k} ((f_k\circ\dots\circ f_1)(x_0)\in [L])<\eps,
\end{multline*}
where $f_i$, $i=1,\dots,k$, are independent and distributed w.r.t.~$\mgr_i$.
\end{prop}

\subsection{Sketch of the proof and structure of the paper}

We start by establishing Theorem~\ref{t:M}, showing that the volume contraction rate has an exponential growth in average. To do so, we use a non-stationary version of additivity of the Furstenberg (or Kullback-Leibler) entropy; in the case when we are starting with the Lebesgue measure, the total entropy provides us with the lower bound on~$\log \mN(F_n)$. Theorem~\ref{t.1} then follows easily once one passes to the projectivized dynamics (on the sphere or on the projective plane): the norm $\|A\|$ can be rewritten in terms of the volume contraction rate~$\mN(f_{A})$ for the projectivized map (see Proposition~\ref{p:N-vs-L}). This is done in Section~\ref{s:entropy}.

Next, in Section \ref{s.LDnorms} we give a proof of Theorem~\ref{t.2}. To prove it, we divide a length~$n$ composition $T_n=A_n\dots A_1$ into a product of groups of matrices of length~$k$, with~$k$ sufficiently large (and chosen depending on the given~$\eps>0$). Now, the log-norm of the product of groups is \emph{almost} the sum of log-norms of these groups, except for the small probability that there is a ``cancelation''; the same applies to the log-length of the image of a given initial unit vector. The sum of the log-norms is a sum of independent random variables with a uniformly bounded exponential moment, thus we do have the large deviations control for these sums. The difficulty here is to control the cancellation effect: we need to show that the norm of a product is quite rarely ``substantially smaller'' than the product of norms.

In order to control the influence of cancellations, we need to control the probability that a long random nonstationary composition sends a given vector to a given direction or, more generally, to a given (hyper)plane. We use Proposition~\ref{p:proj-dissolving} for these estimates, postponing its proof until Section~\ref{s:dissolving}. Theorem~\ref{t.L} follows from Theorem~\ref{t.2} immediately. It also implies Proposition~\ref{t:1.5}, as exponential growth of norms allows to control the changes in the direction of the most contracted direction (we refer to~\cite[Theorem~8.3]{LS} here).

Finally, Section~\ref{s:dissolving} is devoted to the statements on the dissolving of atoms. We start by establishing (in Section~\ref{s:general-dissolve}) Theorem~\ref{p.max.full}. To do so, for an atomic measure we consider a vector given by its atoms weights', and take its $\ell_2$-norm. One-step iteration, that is, passing from $\msp$ to $\mgr*\msp$, corresponds to the averaging of random images~$f_*\msp$. Thus, if such a norm doesn't decrease by a linear factor, the averaged vectors are ``mostly aligned''. Considering the measure with the squared weights, normalizing them, and extracting a convergent subsequence, we find a measure with a deterministic image. This establishes the first part of the proposition. For the second one, we note that if maximal weight of atoms did not converge to zero, the limit measure with the deterministic image also contains an atom. The set of atoms of maximal weight then provides a finite set with a deterministic image.

We conclude by establishing Proposition~\ref{p:proj-dissolving} (in Section~\ref{s:planes}). To do so, we proceed by induction on the dimension $m$ of the subspace to be avoided: the base $m=1$ is given directly by Theorem~\ref{p.max.full}. The induction step is obtained by applying Theorem~\ref{p.max.full} to the induced dynamics on the Grassmanian~$\Gr(m,n)$.

In Appendix~\ref{a.1} we present an example showing that the ``spaces condition'' cannot be omitted from Theorem~\ref{t.L}; see Example~\ref{ex:R-4}.

\section{Furstenberg entropy}\label{s:entropy}

This section is devoted to the proofs of Theorems~\ref{t.1} and~\ref{t:M}. Let us first deduce the former one from the latter. To do so, let us associate to every linear map $A\in \SL(d, \mathbb{R})$ the corresponding projective map $f_A:\RP^{d-1}\to \RP^{d-1}$, where $\RP^{d-1}$ is assumed to be equipped with the standard metric (projected from the sphere~$\bS^{d-1}$). We then have the following

\begin{prop}\label{p:N-vs-L}
For any $A\in \SL(d,\R)$ one has $\mN(f_A)= \|A\|^d$.
\end{prop}
\begin{proof}
Indeed, consider any vector $v$ of unit length, $|v|=1$. Take the hyperplane $\theta_1:=v+\langle v \rangle^{\perp}\subset \R^d$, that is naturally identified to $T_v \bS^{d-1}$, as well as the hyperplane $\theta_2:=Av+\langle Av \rangle^{\perp}$. Consider first the composition of the restriction $A|_{\theta_1}:\theta_1\to A\theta_1$ with the projection from $A\theta_1$ to $\theta_2$ in the direction of~$Av$. The $d-1$-dimensional Jacobian of such a composition is equal to $|Av|^{-1}$ due to the preservation of volume (as $\det A=1$).

Now, take a composition of this map with a radial projection to $\theta_3=T_{Av/|Av|} \bS^{d-1}$; the latter contracts the volume $|A|^{d-1}$ times, thus we finally obtain
\[
\Jac(f|_{A})|_{[v]} = |Av|^{-1}\cdot |Av|^{d-1} = |Av|^{-d},
\]
where $[v]$ is the point of $\RP^{d-1}$ corresponding to the vector~$v$. Hence,
\[
\mN(f_{A}) = \max_{|v|=1} \Jac(f|_{A})|_{[v]}^{-1} = \max_{|v|=1} |Av|^{d} = \|A\|^d.
\]
\end{proof}

\begin{proof}[Deduction of Theorem~\ref{t.1} from Theorem~\ref{t:M}]
Assume that Theorem~\ref{t:M} holds. Now, any probability distribution $\mgr$ on $\SL(d, \mathbb{R})$ induces a probability distribution on the space of projective maps that, slightly abusing the notation, we will also denote by~$\mgr$. The assumptions on Theorem~\ref{t.1} imply the assumptions of Theorem~\ref{t:M} for the projectivized dynamics on $M=\RP^{d-1}$. Thus, for any $n$ and any $\mgr_1,\mgr_2,\dots,\mgr_n \in \bK$ we have
\[
\frac{1}{n} \log \|A_n\dots A_1\| =  \frac{1}{n} \cdot \frac{1}{d} \log \mN(f_{A_n}\dots f_{A_1}) \ge \frac{\qh}{d} =:\lambda>0,
\]
where $\qh$ is provided by the conclusion of Theorem~\ref{t:M}.
\end{proof}

Let us now pass to the preliminaries of the proof of Theorem~\ref{t:M}.

\begin{defi}
Let $\msp$ and $\tilde\msp$ be Borel probability measures on~$M$. Define {\it relative entropy} $h(\msp|\tilde\msp)$ of a measure $\msp$ with respect to a measure $\tilde\msp$ (also known as {\it the Kullback-Leibler information divergence})  as
$$
h(\msp|\tilde\msp)=\left\{
             \begin{array}{ll}
               \int_{M}\log\left(\frac{d\msp}{d\tilde\msp}\right)d\msp, & \hbox{if $\msp\ll\tilde\msp$;} \\
               +\infty, & \hbox{otherwise.}
             \end{array}
           \right.
$$
\end{defi}

It is well known (e.g. see \cite[Section~7]{KV}, or \cite[Lemma~3.1]{Bax}, or \cite[Lemma~2.1]{DV}) that $h(\msp|\msp')$ can be given also as
\begin{equation}\label{eq:h-sup}
h(\msp|\tilde\msp)=\sup_{\psi\in C(M)}\left[-\log\left(\int e^\psi d\tilde\msp\right)+\int\psi d\msp\right];
\end{equation}
this immediately implies that $h(\msp|\tilde\msp)\ge 0$, and (due to convexity of $\log x$) that $h(\msp|\tilde\msp)=0$ if and only if $\msp= \tilde\msp$. Also,~\eqref{eq:h-sup} implies that $h(\msp | \msp')$ is lower-semicontinuous in both variables (w.r.t. the weak convergence), as it can be represented as supremum of a family of continuous functions of~$\msp$ and~$\msp'$.
\begin{defi}\label{d.Furstent}
Given a Borel probability measure $\msp$ on $M$ and a probability distribution $\mgr$ on $\Diff^1(M)$, we define the \emph{Furstenberg entropy} by
\begin{equation}\label{eq:F-e}
\Phi_{\mgr}(\msp) := \int_{\Diff^1(M)} h(f_*\msp| \msp')\, d\mgr(f)\equiv\mathbb{E}_{\mgr} \left( h(f_*\msp| \msp')\right),
\end{equation}
where $\msp '=\mgr*\msp$.
\end{defi}

\begin{remark}
Notice that usually the definition of the Furstenberg entropy is formulated (see e.g.~\cite{Fur1,LL}) in the case when $\msp$ is a stationary measure for a random dynamical system defined by some distribution~$\mgr$. In that case $\msp'=\msp$.
\end{remark}

\begin{lemma}
The Furstenberg entropy $\Phi_{\mgr}(\msp)$ satisfies the following properties:
\begin{enumerate}
\item $\Phi_{\mgr}(\msp)\ge 0$, and $\Phi_{\mgr}(\msp)=0$ if and only if one has $\msp'=f_* \msp$ for $\mgr$-almost every~$f$, where  $\msp'=\mathbb{E}_{\mgr}(f_*\msp)$;
\item $\Phi_{\mgr}(\msp)$ is lower semi-continuous in both variables.
\end{enumerate}
\end{lemma}

\begin{proof}
Both properties directly follow from the properties of $h$. Indeed, one has $h(\msp|\msp')\ge 0$, the equality takes place if and only if $\msp=\msp'$, and this implies the first property. The second one follows from lower semi-continuity of $h(\msp|\msp')$ in both variables.
\end{proof}

\begin{coro}\label{c:nu}
Let us define $\Psi(\mgr):=\inf_{\msp} \Phi_{\mgr}(\msp)$. Then $\Psi(\cdot)$ is lower-semicontinuous, nonnegative, and is equal to $0$ if and only if there are measures $\msp,\msp'$ such that $f_*\msp=\msp'$ for $\mgr$-a.e. $f$.
\end{coro}

\begin{coro}\label{c.positivity}
In our setting (i.e. assuming that the measure condition holds), one has
\[
\qh:=\inf_{\mgr\in \bK}\inf_{\msp}\Phi_{\mgr}(\msp)>0.
\]
\end{coro}
\begin{proof}
Indeed, $\inf_{\msp}\Phi_{\mgr}(\msp)=\Psi({\mgr})$ as a function of $\mgr\in \bK$ is a positive lower-semicontinuous function on a compact metric space, hence
\[
\inf_{\mgr\in \bK}\inf_{\msp}\Phi_{\mgr}(\msp)=\inf_{\mgr\in \bK}\Psi({\mgr})>0.
\]
\end{proof}

Note that given \emph{two} Borel probability measures $\msp,\msp'$ on $M$ and a probability distribution $\mgr$ on $\Diff^1(M)$, one can consider the expectation
\begin{equation}\label{eq:Phi-2}
\Phi_{\mgr}(\msp | \msp') :=\int h(f_*\msp| \msp')d\mgr\equiv\mathbb{E}_{\mgr} \left( h(f_*\msp| \msp')\right).
\end{equation}
The following statement holds:
\begin{lemma}\label{l:average}
\begin{equation}
\Phi_{\mgr}(\msp\mid \msp') = \Phi_{\mgr}(\msp) + h(\mgr * \msp\mid \msp');
\end{equation}
in particular, the Furstenberg entropy, where one substitutes $\msp'=\mgr*\msp$, minimizes $\Phi_{\mgr}(\msp\mid \msp')$ as one varies~$\msp'$.
\end{lemma}
\begin{proof}
Let $\bmsp:=\mgr* \msp$; then, we have
\[
\log\frac{df_*\msp}{d\msp'} = \log\frac{df_*\msp}{d\bmsp} + \log\frac{d\bmsp}{d\msp'},
\]
and thus
\begin{multline*}
\Phi_{\mgr}(\msp\mid \msp') = \E_{\mgr} \int_M \left(\log\frac{df_*\msp}{d\msp'} \right) \, df_*\msp \\
= \E_{\mgr} \int_M \left(\log\frac{df_*\msp}{d\bmsp} \right) \, df_*\msp +
\E_{\mgr} \int_M \left(\log\frac{d\bmsp}{d\msp'} \right) \, df_*\msp  \\
= \E_{\mgr} h(f_*\msp | \bmsp ) +
\int_M \left(\log\frac{d\bmsp}{d\msp'} \right) \, d(\E_{\mgr} f_*\msp) \\
= \Phi_{\mgr}(\msp) + h(\bmsp|\msp').
\end{multline*}
\end{proof}
\begin{remark}
The statement of Lemma~\ref{l:average} can be re-formulated in terms of a random measure $\xi=f_* \msp$ (that is, a random variable taking values in the space of probability measures): it states that for such a random measure one has
\[
\E h(\xi | \msp') = \E h(\xi | \overline{\xi}) + h(\bmsp|\msp'),
\]
where $\overline{\xi}= \E \xi$ is the expectation of $\xi$ (as a vector-valued random variable),
and $\bmsp=\mgr*\msp$.
\end{remark}

Finally, the following additivity property for the nonstationary Furstenberg entropy is a key step of the proof of Theorem~\ref{t:M}:

\begin{prop}[Nonstationary additivity]\label{p:add}
Let $\msp$ be a Borel probability measure on $M$, and $\mgr, \mgr'$ be probability measures on~$\Diff^1(M)$. Then
$$
\Phi_{\mgr'*\mgr}(\msp) = \Phi_{\mgr}(\msp) +\Phi_{\mgr'}(\msp'),
$$
where $\msp'=\mgr*\msp$.
\end{prop}

\begin{proof}
Note that for any measure $\msp_1$ on $M$ and
any $f\in \Diff^1(M)$ one has
\begin{equation}\label{eq:push-f}
h(f_* \msp_1 | \msp'') = h(\msp_1 | f_*^{-1} \msp'').
\end{equation}
Let $\msp''= \mathbb{E}_{\mgr'}(f_*\msp') = \mgr'*\msp'$. For any
$f\in \Diff^1(M)$ one has
\begin{multline*}
\E_{g\sim \mgr} h(f_* (g_*\msp) | \msp'') =\E_{g\sim \mgr} h(g_*\msp |f_*^{-1}  \msp'') =
\Phi_{\mgr}(\msp \mid f_*^{-1}  \msp'')
\\
=
\Phi_{\mgr}(\msp) + h(\msp' \mid f_*^{-1}\msp'' ) = \Phi_{\mgr}(\msp) + h(f_* \msp' \mid\msp'' ) ,
\end{multline*}
where the first and the last equalities are due to~\eqref{eq:push-f}, the second one is the definition of $\Phi_{\mgr}$, and the third one is due to Lemma~\ref{l:average}.

Now, taking an expectation over $f$, distributed w.r.t. $\mgr'$, one gets
\begin{multline*}
\Phi_{\mgr'*\mgr}(\msp) =
\E_{(f,g)\sim \mgr'\times \mgr} h(f_* (g_*\msp) | \msp'')
\\
=  \Phi_{\mgr}(\msp)  +  \E_{f\sim \mgr'} h(f_* \msp' | \msp'') =
 \Phi_{\mgr}(\msp) + \Phi_{\mgr'}(\msp').
\end{multline*}
completing the proof of Proposition \ref{p:add}.
\end{proof}

Now we are in a position to address Theorem~\ref{t:M}.
\begin{proof}[Proof of Theorem~\ref{t:M}]
Let the sequence of measures~$\msp_n$ on~$M$ be defined as
$$
\msp_0:=\Leb, \quad \msp_{j}=\mathbb{E}_{\mgr_{j}}f_*(\msp_{j-1})= \mgr_j*\msp_{j-1}, \quad j=1,2,\ldots
$$
This sequence is an analogue of a sequence of averaged iterations of a measure for the stationary case. Let  $F_n=f_n\circ \ldots\circ f_1$, where $f_i$ are chosen independently w.r.t.~$\mgr_i$. Recall that the maximal volume contraction rate $\mN(F_n)$ can be rewritten in terms of the Radon-Nikodym derivative for the image of the Lebesgue measure:
\[
\mN(F_n) = \max_{x\in M} \Jac(F_n)|_x^{-1} =\max_{x\in M} \left. \frac{d(F_n)_* \msp_0}{\msp_0}\right|_x ;
\]
taking the logarithm and integrating, we have
\[
\log \mN(F_n) \ge h((F_n)_* \msp_0 | \msp_0).
\]
Finally, taking the expectation, we get
\[
\E \log \mN(F_n) \ge \E h((F_n)_* \msp_0 | \msp_0) \ge \E h((F_n)_* \msp_0 | \msp_n) = \Phi_{\mgr_n*\dots*\mgr_1}(\msp_0);
\]
here the second inequality is due to the equality $\E (F_n)_* \msp_0 = \msp_n$ and Lemma~\ref{l:average}.

Now, Proposition~\ref{p:add} implies that
\begin{equation}\label{eq:Phi-n}
\Phi_{\mgr_n*\dots*\mgr_1}(\msp_0)=\sum_{j=1}^n  \Phi_{\mgr_j}(\msp_{j-1}).
\end{equation}
On the other hand, Corollary~\ref{c.positivity} claims that there exists a constant $\qh>0$ such that for all
$\mgr\in\bK_M$ and any measure $\msp$ on $M$,
\[
\Phi_{\mgr}(\msp)\ge \qh;
\]
comparing it to~\eqref{eq:Phi-n}, one gets
\[
\frac{1}{n} \Phi_{\mgr_n*\dots*\mgr_1}(\msp_0) \ge \qh.
\]
Hence,
$$
\frac{1}{n}
\E \log \mN(F_n) \ge
\frac{1}{n} \Phi_{\mgr_n*\dots*\mgr_1}(\msp_0) \ge \qh.
$$
\end{proof}

\section{Large Deviations: norms}\label{s.LDnorms}

In this section we prove Theorem~\ref{t.2}. In the proof we use Proposition~\ref{p:proj-dissolving}; the latter is proven in Section~\ref{s:planes}. Theorem~\ref{t.L} follows from Theorem~\ref{t.2} immediately due to the Borel-Cantelli argument. Finally, we assume throughout this section that the assumptions of Theorem~\ref{t.2} are satisfied.

For a given $k$ (to be chosen later) we decompose the product of matrices of length $n=km$,
$$
T_n=A_n\dots A_1,
$$
into $m$ groups of products of length $k$:
$$
T_n=(A_n\dots A_{k(m-1)+1})\dots (A_k\dots A_1)= B_m\dots B_1,
$$
where
\begin{equation}\label{eq:B-j}
B_j:=(A_{kj}\dots A_{k(j-1)+1}).
\end{equation}

It is not difficult to see that it suffices to establish the conclusion of Theorem~\ref{t.2} for the subsequence~$n_m=km$; we will formally discuss it later, while for the moment limiting our consideration to this subsequence.

We are now going to compare the log-norm $\log\|T_n\|$ with the sum of log-norms of factors, $\sum_{j=1}^m \log \|B_j\|$. Namely, one has
\[
\log\|T_n\| = \log \|B_m\dots B_1\| \le \sum_{j=1}^m \log \|B_j\|.
\]

Now, fix a unit vector $v_0$ and consider the sequence of its intermediate images $T_{jk} v_0=B_j\dots B_1 v_0$. Normalizing these vectors to the unit ones, we get a sequence of unit vectors $v_j$, starting with $v_0$ and recursively defined by
\[
v_j = \frac{B_j v_{j-1}}{|B_j v_{j-1}|}.
\]

Then, the length of the image of $v_0$ after $n$ iterations can be represented as
\[
\log |T_n v_0| = \sum_{j=1}^m \left(\log |T_{jk} v_{0}|-\log |T_{(j-1)k} v_{0}|\right) = \sum_{j=1}^m \log |B_j v_{j-1}|.
\]
Now, define
\begin{equation}\label{eq:R-def}
R_j:=\log \|B_j\| - \log |B_j v_{j-1}|,
\end{equation}
so that
\[
\sum_{j=1}^m \log |B_j v_{j-1}|  = \sum_{j=1}^m (\log \|B_j\| - R_j) = \sum_{j=1}^m \log \|B_j\| - \sum_{j=1}^m R_j.
\]

Finally, the length $|T_n v_0|$ bounds the norm $\|T_n\|$ from below. We thus finally have

\begin{equation}\label{eq:xi-S}
\sum_{j=1}^m \log \|B_j\| \ge \log \|T_n\| \ge \log |T_n v_0| =\sum_{j=1}^m \log \|B_j\| - \sum_{j=1}^m R_j.
\end{equation}

Let $\xi_j:=\log\|B_j\|$; as the products $B_j$ are independent (as random variables), so are random variables~$\xi_j$. Moreover, their exponential moments $\E e^{\gamma \xi_j}$ are uniformly bounded due to the finite moment condition~\eqref{eq:finite-moment}: they do not exceed~$C^{k}$. Thus, a standard Large Deviations Theorem is applicable to them:

\begin{lemma}\label{l:xi-LD}
For any $k$ and any $\eps'>0$ there exists $\delta'>0$, $C'$ such that
$$
\Prob(|\xi_1+\dots+\xi_m-\E(\xi_1+\dots+\xi_m)|>\eps' m) < C' e^{-\delta' m}.
$$
\end{lemma}
Unfortunately, for the statement in this form we could not find an exact reference (in most references the random variables are assumed either to be uniformly bounded, or identically distributed). Thus (even though the technique is very well-known), for the reader's convenience, we provide here a (standard) proof:
\begin{proof}
It suffices to use the exponential moment method and Chernoff bounds. Namely, let $\eps'>0$ be given. For any $t<\gamma$, and any random variable $\xi$ satisfying~\eqref{eq:finite-moment}, consider the expectation $\varphi_{\xi}(t):=\E e^{t\xi}$. Note that for $x>0$
\[
e^x \le 1+ x + \frac{x^2}{2} e^x;
\]
substituting $x=t\xi$ and taking the expectation, one gets
\begin{equation}\label{eq:e-xi}
\varphi_{\xi}(t)=\E e^{t\xi} \le 1 + t\cdot \E \xi + t^2 \cdot \E (\xi^2 e^{t\xi}).
\end{equation}
Now, the bound $\E e^{\gamma \xi}<C^{k}$ implies that the expectation $\E(\xi^2 e^{t\xi})$ is uniformly bounded for $t\in(0,\frac{\gamma}{2})$ by some (explicit) constant~$c$. Thus, the right hand side of~\eqref{eq:e-xi} is bounded from above by $1+t(\E\xi + tc)$. Choosing $t>0$ sufficiently small so that $tc<\frac{\eps'}{2}$, one gets
\[
\varphi_{\xi}(t) \le 1+ t \cdot(\E \xi + tc) \le 1+t\cdot(\E \xi + \frac{\eps'}{2}) < e^{t(\E \xi + \frac{\eps'}{2})}.
\]
Fix such $t$; the Markov inequality then implies that
\begin{multline*}
\Prob\left(\sum_{j=1}^m \xi_j - \sum_{j=1}^m \E \xi_j > m\eps'\right) =
\Prob\left(e^{\sum_{j=1}^m \xi_j}  > e^{\sum_{j=1}^m (\E \xi_j +\eps')} \right) \le
\\
 \le \frac{\E e^{\sum_{j=1}^m \xi_j}}{e^{\sum_{j=1}^m (\E \xi_j +\eps')} } =  \prod_{j=1}^m \frac{\varphi_{\xi_j}(t)}{e^{t(\E \xi_j +\eps')}} < (e^{-t \frac{\eps'}{2}})^m=e^{-m\cdot t \frac{\eps'}{2}}
\end{multline*}
The estimate from below is obtained in the same way by considering small negative~$t$.
\end{proof}

At the same time, a choice of a sufficiently large $k$ allows to establish large deviations-type bound for $R_j$'s. Namely, the following statement holds:
\begin{prop}\label{p:R-LD}
For any $\eps''>0$ there exists $k_1=k_1(\eps'')$, such that for any $k>k_1$ for some $\delta''>0,C''$ one has for all $n=km$
$$
\Prob(R_1+R_2+\dots+R_{m} > n\eps'') < C'' e^{-\delta'' m}.
$$
\end{prop}

Now we need to justify Proposition~\ref{p:R-LD}.
\begin{proof}[Proof of Proposition~\ref{p:R-LD}]
Let us first rewrite the definition of the random variables~$R_j$. Namely, for a given matrix $B\in \SL(d,\R)$ and a given nonzero vector $u$, consider the log-difference between the norm of $B$ and how the application of $B$ scales the vector $u$:
\begin{equation}\label{e.theta}
\Theta(B,u):=\log \|B\| - (\log |Bu| - \log |u|) =\log \|B\| - \log\left| B\frac{u}{|u|}\right|.
\end{equation}
Then,
$$
R_j=\Theta(B_j,v_{j-1}).
$$

In order to prove Proposition~\ref{p:R-LD}, we will provide upper bounds for these random variables in a way that would be suitable for Large Deviations type bounds. To do so, we divide the product $B$ of (yet unknown) length $k$ into two parts of lengths $k'$ and $k''=k-k'$. The ``relatively short'' part of length $k'$, consisting of matrices that are applied first, will serve to ``randomize'' the image of a vector, so that the ``long'' part of length $k''$ has a  high probability to expand the resulting vector ``almost as strongly as it can''. Namely, let
\[
B_j = A_{kj} \dots A_{k(j-1)+1} = \underbrace{A_{kj} \dots A_{k(j-1)+k'+1}}_{\Bi_j}
\underbrace{A_{kj+k'} \dots A_{k(j-1)+1}}_{\Bq_j} = \Bi_j \Bq_j.
\]
Note that for any unit vector $v$ and any two matrices $\Bi,\Bq\in\SL(d,\R)$,
\begin{multline*}
\Theta(\Bi\Bq,v)=\log \|\Bi\Bq\| - \log |\Bi\Bq v| \\
 \le \log \|\Bi\| +\log\|\Bq\| - \log\left|\Bi \frac{\Bq v}{|\Bq v|} \right| - \log |\Bq v|
 \\ = \Theta(\Bi,\Bq v) + \Theta(\Bq,v).
\end{multline*}
Moreover, one has
\[
\Theta(\Bq,v)=\log \|\Bq\|- \log |\Bq v| \le \log \|\Bq\| + \log \|\Bq^{-1}\| \le d \log \|\Bq\|.
\]

Hence, the sum of $R_j$'s can be bounded from above by
\begin{equation}\label{eq:R-Theta}
\sum_{j=1}^m R_j = \sum_{j=1}^m \Theta(B_j,v_{j-1}) \le \sum_{j=1}^m \Theta(\Bi_j,\Bq_j v_{j-1}) +
d \sum_{j=1}^m\log \| \Bq_j\|.
\end{equation}

Now, we have the following conditional estimate:
\begin{lemma}\label{l:B-L}
For any matrix $\Bi \in \SL(d,\R)$ there exists a hyperplane $L\in \Gr(d-1,d)$, such that for any $\rho>0$ and any unit vector~$u$ with angle at least $\rho$ with~$L$,
\[
\Theta(\Bi,u) \le |\log \sin \rho|.
\]
\end{lemma}
\begin{proof}
Indeed, the matrix $\Bi$ can be represented as
\[
\Bi = O_1 \left(\begin{smallmatrix} \|\Bi\| & 0 \\ 0 & *
\end{smallmatrix}
\right)
 O_2,
\]
where $O_1,O_2$ are orthogonal transformations. We take $L_0$ to be the coordinate hyperplane $\{y_1=0\}$, and let $L:= O_2^{-1} L_0$. Then, if the angle between a unit vector~$u$ and $L$ is at least $\rho$, so is the angle between $O_2 u$ and $L_0$, and hence the absolute value of the first coordinate $y_1$ of $O_2 u$ is at least $\sin \rho$. Thus,
\[
|\Bi u| = \left| \left(\begin{smallmatrix} \|\Bi\| & 0 \\ 0 & *
\end{smallmatrix}
\right)
 O_2 u \right| = \left| \left(\begin{smallmatrix} \|\Bi\| y_1 \\  *
\end{smallmatrix}
\right) \right| \ge \|\Bi\| \sin \rho;
\]
the desired upper bound follows from the definition of~$\Theta(\Bi,u)$.
\end{proof}

Let $\eps''>0$ from the statement of Proposition~\ref{p:R-LD} be given.
Take $\eps_1>0$ sufficiently small so that
\begin{equation}\label{eq:eps-1-cond}
C\eps_1<\frac{\eps''\gamma}{12d},
\end{equation}
where the constant $C$ is the upper bound from~\eqref{eq:finite-moment}; the reason for this choice will become clear later.
By Proposition~\ref{p:proj-dissolving} there exists $k'$ such that for any point $x_0\in \RP^{d-1}$, any hyperplane $L\in \Gr(d-1,d)$ and any $\mgr_1,\dots,\mgr_{k'}\in \bK$, one has an upper bound for the probability
\[
\Prob_{\mgr_1*\dots*\mgr_{k'}} (F(x_0)\in [L]) < \frac{\eps_1}{2}.
\]
This probability can also be written as
\[
(\mgr_1*\dots*\mgr_{k'}*\delta_{x_0})([L]);
\]
the set $\bK^{k'} \times \RP^{d-1} \times \Gr(d-1,d)$ is compact, and due to the semi-continuity arguments thus there exists a positive $\rho>0$ such that
\begin{multline}\label{eq:rho-kp}
\forall \mu_1,\dots,\mu_{k'} \quad \forall x_0\in \RP^{d-1} \quad \forall L\in \Gr(d-1,d) \\
(\mgr_1*\dots*\mgr_{k'}*\delta_{x_0})(U_{\rho}([L]))<\eps_1.
\end{multline}
Fix such $\rho>0$; let us show that the conclusions of the proposition are satisfied for any $k$ large enough to ensure that
\begin{equation}\label{eq:k-choice}
\frac{|\log \sin \rho|}{k} < \frac{\eps''}{3} \quad \text{and} \quad \frac{d \log (eC)}{\gamma} \cdot  \frac{k'}{k} < \frac{\eps''}{3}.
\end{equation}

Indeed, assume that $k$ is chosen so that~\eqref{eq:k-choice} holds. For each $j=1,\dots,m$, let the hyperplane $L_j\in \Gr(d-1,d)$ correspond to $\Bi_j$ by Lemma~\ref{l:B-L}. We then decompose the sum in the right hand side of~\eqref{eq:R-Theta}, depending on whether the angle between the vector $\Bq_j v_{j-1}$ and the hyperplane $L_j$ is larger or smaller than~$\rho$; it is the latter case that corresponds to the ``strong cancellation''.
\begin{multline}\label{eq:R-3}
R_1+\dots+R_m \le \sum_{j=1}^m \Theta(\Bi_j,\Bq_j v_{j-1}) +
d \sum_{j=1}^m\log \| \Bq_j\| \\
\le m |\log \sin \rho|  + \sum_{j=1}^m \Theta(\Bi_j,\Bq_j v_{j-1}) \Ind_{\dist([\Bq_j v_{j-1}], [L_j])< \rho} + d \sum_{j=1}^m\log \| \Bq_j\| \le
\\
\le m |\log \sin \rho|  + d \sum_{j=1}^m \log \|\Bi_j \| \cdot \Ind_{\dist([\Bq_j v_{j-1}], [L_j])< \rho} + d \sum_{j=1}^m\log \| \Bq_j\|
\end{multline}
Let us show that each of the three summands in the right hand side of~\eqref{eq:R-3} is larger than $\frac{\eps''}{3}n$ with the exponentially small probability. Namely, the first one is deterministic, and we have
\[
m |\log \sin \rho| = n \, \frac{|\log \sin \rho|}{k} < \frac{\eps''}{3} n
\]
due to the first condition in~\eqref{eq:k-choice}. The last summand is the contribution of $k'$ out of each $k$ multiplied matrices.
Repeating the standard proof of the Large Deviations Theorem, we have
\[
\E \exp\left(\gamma \sum_{j=1}^m \log \|\Bq_j\|\right) = \E \left( \prod_{j=1}^m  \|\Bq_j\|^{\gamma} \right) \le C^{mk'}.
\]
Thus, by Markov inequality,
\[
\Prob \left(\gamma \sum_{j=1}^m \log \|\Bq_j\| \ge mk' \log(eC) \right)
=\Prob \left(\prod_{j=1}^m  \|\Bq_j\|^\gamma \ge C^{mk'}e^{mk'} \right)
\le e^{-mk'}=e^{-n \frac{k'}{k}},
\]
and if the event in the left hand side doesn't take place, we have
\[
d \sum_{j=1}^m \log \|\Bq_j\| < \frac{d}{\gamma} mk' \log(eC) = \frac{d k' \log (eC)}{\gamma k} \, n < \frac{\eps''}{3} n,
\]
where the last inequality is due to the second inequality in~\eqref{eq:k-choice}. Finally, let us estimate the second summand in the right hand side of~\eqref{eq:R-3}. To do so, denote
\[
\eta_j:=\log \|\Bi_j \| \cdot \Ind_{\dist([\Bq_j v_{j-1}], [L_j])< \rho},
\]
so that this summand has the form $d\sum_{j=1}^m \eta_j$. Consider the sequence of expectations
\[
\mI_j:= \E \exp \left(\frac{\gamma}{k''} \sum_{i=1}^j \eta_i \right),
\]
and recall that $\rho$ and $k'$ were chosen for~\eqref{eq:rho-kp} to hold. We have the following upper bound:
\begin{lemma}\label{l:ind}
For any $j=1,\dots,m$ one has
$\mI_j\le (1+C \eps_1) \mI_{j-1}$, where $\mI_0:=1$.
\end{lemma}
\begin{proof}
Note first that we have the following upper bound for $\eta_j$
\begin{multline}\label{eq:exp-1}
\exp\left(\frac{\gamma}{k''} \eta_j\right)=  \exp\left(\frac{\gamma}{k''} \log \|\Bi_j \| \cdot \Ind_{\dist([\Bq_j v_{j-1}], [L_j])< \rho}\right) \\
\le 1+ \exp\left(\frac{\gamma}{k''} \log \|\Bi_j \|\right) \cdot \Ind_{\dist([\Bq_j v_{j-1}], [L_j])< \rho}
\\ = 1+ \|\Bi_j \|^{\gamma/k''} \cdot \Ind_{\dist([\Bq_j v_{j-1}], [L_j])< \rho}
\end{multline}

Let $\mF_j$ be the $\sigma$-algebra generated by $\Bi_j$ and all $A_i$, $i=1,\dots, k(j-1)$. Then~$\|\Bi_j\|$, the hyperplane $L_j$ and the vector $v_{j-1}$ are $\mF_j$-measurable, while $\Bq_j$ is independent from $\mF_j$. One can re-write $\mI_j$ as
\begin{multline}\label{eq:E-product}
\mI_j=\E \exp\left(\frac{\gamma}{k''} \sum_{i=1}^j \eta_i\right) = \E \left( \E \left(\left. \exp(\frac{\gamma}{k''} \sum_{i=1}^j \eta_i) \,\right| \, \mF_j \right)\right)
\\ = \E \left( \exp(\frac{\gamma}{k''} \sum_{i=1}^{j-1} \eta_i) \cdot \E \left( \left.\exp(\frac{\gamma}{k''} \eta_j) \,\right|\, \mF_j \right)\right).
\end{multline}
Now, due to~\eqref{eq:exp-1},
\begin{multline}
\E\left(\left.\exp(\frac{\gamma}{k''} \eta_j)  \,\right|\,  \mF_j\right) \le 1+ \E( \|\Bi_j\|^{\gamma/k''} \Ind_{\dist([\Bq_j v_{j-1}], [L_j])< \rho} \mid \mF_j) \\ = 1+\|\Bi_j\|^{\gamma/k''} \Prob (\dist([\Bq_j v_{j-1}], [L_j])< \rho) \le 1+\|\Bi_j\|^{\gamma/k''}\eps_1.
\end{multline}
The right hand side is independent from all $A_i$, $i=1,\dots, k(j-1)$, thus plugging this into~\eqref{eq:E-product}, we get the desired
\[
\mI_j \le \mI_{j-1}\cdot (1+ \eps_1 \cdot \E \|\Bi_j\|^{\gamma/k''}) < \mI_{j-1}\cdot (1+C \eps_1);
\]
here we used the inequality~$\E \|\Bi_j\|^{\gamma/k''}<C$, that follows from the moment condition~\eqref{eq:finite-moment}.
\end{proof}

Now, applying the induction argument, and using the inequality $1+x\le e^x$, from Lemma~\ref{l:ind} we get
\[
\mI_m=\E \exp\left(\frac{\gamma}{k''} \sum_{i=1}^m \eta_i\right) \le (1+C\eps_1)^m < e^{C\eps_1 m} = e^{C\eps_1 \frac{n}{k}}
\]

In particular, by Markov inequality, we have
\begin{multline}
\Prob\left( d\sum_{j=1}^m \eta_j \ge \frac{\eps''}{3} n\right) = \Prob\left( \exp\left(\frac{\gamma}{k''}\sum_{j=1}^m \eta_j \right) \ge \exp\left( \frac{\gamma}{dk''} \cdot \frac{\eps''}{3} n\right)\right) \\
\le \frac{\mI_m}{\exp(\frac{\gamma/k''}{d}\cdot \frac{\eps''}{3} n)} < \exp \left( (C\eps_1 - \frac{\gamma \eps''}{3d}) \cdot \frac{n}{k} \right) = \exp(-\delta''_1 n),
\end{multline}

where
\[
\delta''_1:= \frac{1}{k} (\frac{\gamma \eps''}{3d}-C\eps_1)
\]
is positive due to the choice of $\eps_1$.
\end{proof}

We are now ready to complete the proof of Theorem~\ref{t.2}.

\begin{proof}[Proof of Theorem~\ref{t.2}]
Let $\eps>0$ be given. Joining~\eqref{eq:xi-S} and Proposition~\ref{p:R-LD} for $\eps''=\frac{\eps}{6}$, we get that there exist $k, \delta'', C''$ such that for $n=km$
\[
\sum_{j=1}^m \xi_j \ge \log \|T_n\| \ge \log |T_n v_0| \ge \sum_{j=1}^m \xi_j - \frac{\eps}{6} n.
\]
with the probability at least $1-C''e^{-\delta''m}=1-C''e^{-\frac{1}{k}\delta''n}$.

Fix such $k$. Now, from Lemma~\ref{l:xi-LD} with $\eps'=\frac{\eps k}{6}$, we have that
\[
\Prob\left(\Big|\sum_{j=1}^m \xi_j- \sum_{j=1}^m \E(\xi_j) \Big|>\frac{\eps}{6} n\right) < C' e^{-\delta' m}= C' e^{-\frac{1}{k}\delta' n}.
\]
Hence, taking $\delta'''=\frac{1}{k}\min(\delta',\delta'')$, $C'''=C'+C''$, we get that
\begin{equation}\label{eq:T-LD}
\Prob\left(\left |\log \|T_n\| - \sum\nolimits_{j=1}^{m} \E (\xi_j) \right| > \frac{\eps}{3} n \right) < C''' e^{-\delta''' n}.
\end{equation}
and
\begin{equation}\label{eq:T-v-LD}
\Prob\left(\left |\log |T_n v_0 | - \sum\nolimits_{j=1}^{m} \E (\xi_j) \right| > \frac{\eps}{3} n \right) < C''' e^{-\delta''' n}.
\end{equation}
Now, let us estimate the difference between $L_n=\E \log \|T_n\|$ and $\sum_{j=1}^{m} \E (\xi_j)$:
\begin{multline}\label{eq:shift-L}
\left| L_n -\sum\nolimits_{j=1}^{m} \E (\xi_j) \right| \le \E \left| \log \|T_n\| - \sum\nolimits_{j=1}^{m} \E (\xi_j) \right|
\\
\le \frac{\eps}{3} n + \E \left( \left| \log \|T_n\| - \sum\nolimits_{j=1}^{m} \E (\xi_j) \right| \cdot \Ind_{\left| \log \|T_n\| - \sum\nolimits_{j=1}^{m} \E (\xi_j) \right| > \frac{\eps}{3} n}\right).
\end{multline}
Note that due to the assumption~\eqref{eq:finite-moment} the expectations $\E \log \|A_j\|$ are uniformly bounded:
\begin{lemma}\label{l:l-A}
There exists $C_A$ such that for any $\mgr\in \bK$
\[
\int \log \|A\| \, d\mgr(A) \le C_A, \quad \int \log^2 \|A\| \, d\mgr(A) \le C^2_A.
\]
\end{lemma}
\begin{proof}
It suffices to note that the quotients $\frac{\log x}{x^{\gamma}}$ and $\frac{\log^2 x}{x^{\gamma}}$ are uniformly bounded on $[1,\infty)$, thus so are the quotients $\frac{\log \|A\|}{\|A\|^{\gamma}}$ and $\frac{\log^2 \|A\|}{\|A\|^{\gamma}}$ for $A\in \SL(d,\R)$.
\end{proof}
Lemma~\ref{l:l-A} and the sub-additivity of the logarithm of the norm imply that
\[
\left|  \sum\nolimits_{j=1}^{m} \E (\xi_j) \right| \le C_A n, \quad \E \log \|T_n\|  \le C_A n,
\]
\[
\E ( \log \|T_n\| )^2 \le \E( \sum\nolimits_{i=1}^{n} \log \|A_i\| )^2 \le C_A^2 n^2,
\]
\[
\E \left| \log \|T_n\| - \sum\nolimits_{j=1}^{m} \E (\xi_j) \right|^2 \le 4C_A^2 n^2.
\]
Applying the Cauchy inequality for the second summand in the right hand side in~\eqref{eq:shift-L} then gives
\begin{multline}\label{eq:T-tails}
 \E \left( \left| \log \|T_n\| - \sum\nolimits_{j=1}^{m} \E (\xi_j) \right| \cdot \Ind_{\left| \log \|T_n\| - \sum\nolimits_{j=1}^{m} \E (\xi_j) \right| > \frac{\eps}{3} n}\right)  \le
 \\ \le 2C_A n \cdot (C''')^{1/2} e^{-\frac{1}{2}\delta''' n} < \const;
\end{multline}
in particular, the expectation in~\eqref{eq:T-tails} does not exceed $\frac{\eps}{6}n$ for all $n$ sufficiently large, and thus from~\eqref{eq:shift-L}
\[
\left| L_n -\sum\nolimits_{j=1}^{m} \E (\xi_j) \right| \le \frac{\eps}{3} n+ \frac{\eps}{6}n =\frac{\eps}{2} n.
\]

 Hence,~\eqref{eq:T-LD} and~\eqref{eq:T-v-LD} imply respectively for all $n=km$ sufficiently large
\begin{equation}\label{eq:T-LD-L}
\Prob \left( \left |\log \|T_n\| - L_n \right|> \frac{5\eps}{6} n \right) <C''' e^{-\delta''' n},
\end{equation}
\begin{equation}\label{eq:T-v-LD-L}
\Prob \left( \left |\log |T_n v_0| - L_n \right|> \frac{5\eps}{6} n \right) <C''' e^{-\delta''' n},
\end{equation}
thus establishing (upon taking any positive $\delta<\delta'''$) the conclusions of Theorem~\ref{t.2} for such~$n$.

For the case of a general $n$, let us write it as $n=km+i$, where $0\le i\le k-1$, and write
\[
T_n = A_{km+i} \dots A_{km+1} T_{km} = \tB \, T_{km},
\]
where $\tB:=A_{km+i} \dots A_{km+1}$. Let us estimate the norm $\|\tB\|$ and the difference $|L_n-L_{km}|$.
For the latter one, note that the increments in the sequence $L_n$ are uniformly bounded:
\begin{lemma}\label{l:L-increments}
For any sequence $\mgr_1,\mgr_2,\dots\in\bK$ and any $n$ one has
\begin{equation}
|L_n-L_{n+1}|\le (d-1)C_A.
\end{equation}
\end{lemma}
\begin{proof}
$T_{n+1}=A_{n+1} T_n$, using the sub-multiplicativity of the norm, we get
\[
-\log \|A_{n+1}^{-1}\| \le \log \|T_{n+1}\|- \log \|T_n\| \le \log \|A_{n+1}\|.
\]
Using the inequality $\log \|A^{-1}\|\le (d-1) \log \|A\|$, and taking the expectation, we get the desired uniform bound
\[
-(d-1) C_A\le L_{n+1}-L_n \le C_A.
\]
\end{proof}
In particular, for all $n$ sufficiently large we have $|L_{n}-L_{km}|\le \frac{\eps}{12}n$, where $n=km+i$, and $i<k$.

Now, note that for $\tB$ as a product of at most $(k-1)$ independent matrices $A_j$, one has
\[
\E \| \tB\|^\gamma \le C^{k-1}.
\]
thus implying that
\[
\P\left( \log \| \tB \| > \frac{\eps}{12} n\right) = \P\left( \| \tB \|^{\gamma} > e^{\gamma\frac{\eps}{12} n}\right) \le C^{k-1} e^{-\gamma \frac{\eps}{12} n} = C_1 e^{-\delta_1 n},
\]
where $\delta_1:=\frac{\eps\gamma}{12}$.
Joining these estimates together with~\eqref{eq:T-LD-L} and~\eqref{eq:T-v-LD-L} provides the conclusion of Theorem~\ref{t.2} for a general (sufficiently large)~$n$.
\end{proof}
Notice that the standard application of Borel-Cantelli argument now shows that Theorem~\ref{t.2} implies Theorem~\ref{t.L}.

We are now ready to conclude the proof of Proposition~\ref{t:1.5}.
\begin{proof}[Proof of Proposition~\ref{t:1.5}] We start by recalling that for the dimension $d=2$, the spaces condition is implied by the measures one (see Remark~\ref{rq:notes}). Hence, all the assumptions of Theorems~\ref{t.L} and~\ref{t.2} are satisfied, and thus the conclusions of these theorems hold.

We will use Theorem~8.3 from~\cite{LS}; for the reader's convenience, we recall here its statement:
\begin{theorem*}[{\cite[Theorem~8.3]{LS}}]
Assume that one is given a sequence of matrices $A_i\in \SL(2,\R)$, such that
\begin{itemize}
\item One has \begin{equation}\label{eq:T-A-conv}
\sum\nolimits_{n=1}^{\infty}  \frac{\|A_{n+1}\|^2}{\|T_n\|^2} < \infty,
\end{equation}
where $T_n=A_n\dots A_1$;
\item for some monotone increasing function $f(n)$
\begin{equation}\label{eq:T-A-f}
\lim_{n\to\infty} \frac{\ln \|T_n\|}{f(n)} =1, \quad  \lim_{n\to\infty} \frac{\ln \|A_n\|}{f(n)} =0;
\end{equation}
\item for any $\eps>0$ one has
\begin{equation}\label{eq:exp-eps-f}
\sum_{n=1}^{\infty} e^{-\eps f(n)} <\infty.
\end{equation}
\end{itemize}
Then there exists a unit vector $u_{\infty}\in\R^2$ such that
\begin{equation}\label{eq:T-u}
\lim_{n\to\infty} \frac{\ln |T_n u_{\infty}|}{f(n)} =-1.
\end{equation}
\end{theorem*}

Note that the function $f(n)$ is required in this theorem to be monotonously increasing.
We thus choose
\begin{equation}\label{eq:f-def}
f(n):=\min_{m\ge n} (L_m - \frac{\lambda}{2} (m-n)),
\end{equation}
where $\lambda>0$ is given by Theorem~\ref{t.1}.
We will show (see Corollary~\ref{c:L-growth} below) that this function is increasing and that the difference $|L_n-f(n)|$ is uniformly bounded. Once such a statement is established, the other assumptions are easily verified.

Indeed, the lower bound $L_n\ge n\lambda$ implies the convergence of the series~\eqref{eq:exp-eps-f}. Now, for a random sequence $A_i$ and the corresponding sequence $T_n$, the first part of the condition~\eqref{eq:T-A-f} almost surely holds due to Theorem~\ref{t.L}.

Meanwhile, for any $\eps>0$ the probability of the event $\left\{\log \|A_n\| \ge n\eps\right\}$ is bounded from above by $Ce^{-\gamma \eps n}$ due to the moments assumption~\eqref{eq:finite-moment}. Due to the Borel--Cantelli argument we have almost surely $\log \|A_n\|=o(n)$, what implies the second part of the condition~\eqref{eq:T-A-f}. Finally, having $\log \|A_n\|=o(n)$ and $\log \|T_n\|\sim L_n$, together with the lower bound $L_n\ge \lambda n$, easily implies the convergence of the series~\eqref{eq:T-A-conv}.

All the assumptions of~\cite[Theorem~8.3]{LS} are thus almost surely satisfied; hence, its conclusion holds, and one can take $\bv=u_{\infty}$. Indeed,~\eqref{eq:T-u} implies that
\[
\lim_{n\to\infty} \frac{\ln |T_n \bv| + L_n}{L_n} =0,
\]
and the upper bound $L_n=O(n)$ then implies the conclusion of Proposition~\ref{t:1.5}.

Let us show that the difference $L_n-f(n)$ is uniformly bounded. Denote for $m\ge n$
\[
T_{[n,m]}:=T_m T_n^{-1} = A_m\dots A_{n+1},
\]
and let
\[
L_{[n,m]}:= \E \log \|T_{[n,m]}\|.
\]
We then have the following
\begin{lemma}\label{l:L-add}
For any $\eps>0$ there exists $C_{\eps}$ such that
\[
\forall n \,\, \forall m\ge n \quad L_m \ge L_n + (1-\eps) L_{[n,m]} - C_{\eps}.
\]
\end{lemma}
\begin{proof}
As in the proof of Proposition~\ref{p:R-LD}, assuming that $\eps>0$ is given, we take a sufficiently small $\eps_1$, namely,
\[
\eps_1:= \frac{\eps}{d}.
\]
Then, from Proposition~\ref{p:proj-dissolving} it follows that there exists $k'$ such that for any point $x_0\in \RP^{d-1}$, any hyperplane $L\in \Gr(d-1,d)$ and any $\mgr_1,\dots,\mgr_{k'}\in \bK$, one has an upper bound on the probability
\[
\Prob_{\mgr_1*\dots*\mgr_{k'}} (F(x_0)\in [L]) < \frac{\eps_1}{2}.
\]
Hence, there exists a positive $\rho>0$ such that
\begin{multline*}
\forall \mu_1,\dots,\mu_{k'} \quad \forall x_0\in \RP^{d-1} \quad \forall L\in \Gr(d-1,d) \\
(\mgr_1*\dots*\mgr_{k'}*\delta_{x_0})(U_{\rho}([L]))<\eps_1.
\end{multline*}
Fix such $k'$ and such $\rho$; once $k'$ is fixed, we can restrict ourselves to $m>n+k'$, as for $m\le n+ k'$ the conclusion of the lemma is satisfied for any $C_{\eps}>d C_A k'$ due to inequalities
\[
L_{[n,m]} \le  C_A (m-n), \quad L_m\ge L_n - (d-1) C_A (m-n).
\]
Denote then
\[
\Bq:=T_{[n,n+k']}, \quad \Bi:=T_{[n+k',m]}.
\]
Let the hyperplane $L\in \Gr(d-1,d)$ correspond to $\Bi$ in terms of Lemma~\ref{l:B-L}, and choose $v$ to be the (random) unit vector such that $\|T_n\|=|T_n v|$. The log-norm of the product
\[
T_m = \Bi \Bq T_n
\]
can be then estimated from below by
\begin{equation}\label{eq:lower-rho}
\log \|T_m \| \ge \log \|\Bi\| - |\log \sin \rho| -(d-1) \log \|\Bq\| + \log \|T_n\|
\end{equation}
if the angle between $L$ and $\Bq T_n v$ is at least~$\rho$, and by
\begin{equation}\label{eq:lower-not}
\log \|T_m \| \ge -(d-1)\log \|\Bi\| -(d-1) \log \|\Bq\| + \log \|T_n\|
\end{equation}
otherwise.
Taking the conditional expectation w.r.t. $T_n$ and $\Bi$, and using the fact that the conditional probability of~\eqref{eq:lower-not} is at most $\eps_1$, one gets
\[
\E (\log \| T_m \| \, \mid T_n, \Bi) \ge (1- d\eps_1) \log \|\Bi\| - |\log \sin \rho| - (d-1) \E\log \|\Bq\| + \log \|T_n\|.
\]
Joining this with the estimate
\[
\log \|\Bi\| = \log \|T_{[n,m]}\Bq^{-1}\| \ge \log \|T_{[n,m]}\| - (d-1) \log \|\Bq\|,
\]
and recalling that $d\eps_1=\eps$, we get
\begin{multline*}
L_m=\E \log \| T_m \| =
\E( \E (\log \| T_m \| \, \mid T_n, \Bi) ) \ge
\\
\ge \E \log \|T_n\| + (1- \eps) \, \E\!  \log \|T_{[n,m]}\| - |\log \sin \rho| - 2(d-1) \E\log \|\Bq\| =
\\
= L_n + (1- \eps) L_{[n,m]} - (|\log \sin \rho| + 2(d-1) \E\log \|\Bq\|).
\end{multline*}
Taking
\[
C_{\eps}:= |\log \sin \rho| + 2(d-1) \E\log \|\Bq\|
\]
completes the proof.
\end{proof}

\begin{coro}\label{c:L-growth}
The sequence $f(n)$, defined by~\eqref{eq:f-def}, is increasing, and the sequence of differences $|f(n)-L_n|$ is uniformly bounded.
\end{coro}
\begin{proof}
Directly from the definition one has
\[
f(n+1)-f(n)\ge \frac{\lambda}{2}>0,
\]
thus implying that $f(n)$ is increasing. On the other, let us apply Lemma~\ref{l:L-add} with $\eps=\frac{1}{2}$; we have that $L_{[n,m]}\ge \lambda (m-n)$, and thus
\[
L_m - \frac{\lambda}{2} (m-n) \ge L_n  - C_{\eps}.
\]
Hence, $L_n \ge f(n) \ge L_n - C_{\eps}$.
\end{proof}
Corollary~\ref{c:L-growth} allows to use Theorem~8.3 from~\cite{LS} and therefore completes the proof of Proposition~\ref{t:1.5}.
\end{proof}

\section{Dissolving the atoms}\label{s:dissolving}
\subsection{General case: a point avoids a point}\label{s:general-dissolve}

This section is devoted to the proof of Theorem~\ref{p.max.full}.

 First, let us introduce some notations. Given measures $\msp$ on $X$ and $\mgr\in \mathbf{K}_X$, denote $\msp'=\mathbb{E}_{\mgr}f_*\msp=\mgr*\msp$.
 Let $x_1,x_2,\dots$ be the atoms of $\msp$ of weights $m_1,m_2,\dots,$ respectively. In the same way, let $y_1,y_2,\dots$ be the atoms of $\msp'$ of weights $m_1',m_2',\dots,$ respectively. Let $p_{i,j}$ be the $\mgr$-probability that the atom $x_i$ is sent to the atom $y_j$. Finally, we formally set $p_{0,j}$ to be the probability that the preimage of $y_j$ is not an atom, and denote $m_0:=0$. It is easy to see that for any $j=1, 2, \ldots$ we have
 $$
 \quad m_j'= \sum_{i\ge 0} p_{i,j} m_i, \quad \sum_{i\ge 0} p_{i,j} =1,
 $$
 and for any $i=1, 2, \ldots$ we have
 $$
 \quad \sum_{j\ge 1} p_{i,j} \le 1.
 $$

Consider then the sums of squares of these weights
$$
\mE:=\sum_i m_i^2, \quad \mE':=\sum_j m_j'^2.
$$
Consider also the variances
$$
D_j:=\sum_{i\ge 0} p_{i,j} (m_i - m_j')^2.
$$
Notice that $D_j$ is indeed the variance of a random variable $\rxi_j$ that takes the value~$m_i$ with the probability $p_{i,j}$. Then, the standard identity implies that
$$
m_j'^2 + D_j = (\E \rxi_j)^2 + Var(\rxi_j) = \E (\rxi_j^2) = \sum_i p_{i,j} m_i^2
$$
Summing over $j$, we get
$$
\mE'+ \sum_j D_j = \sum_j \sum_i p_{i,j} m_i^2 =\sum_i \left(\sum_j p_{i,j}\right) m_i^2 \le \sum_i m_i^2 =\mE.
$$
Denoting $D:=\sum_j D_j$, we thus get
\begin{equation}\label{e.zd}
\mE\ge
\mE'+D.
\end{equation}
Consider also the non-probability measures
$$
\rho:=\sum_i m_i^2 \delta_{x_i}, \quad \rho':=\sum_j m_j'^2 \delta_{y_j}.
$$
The following statement holds:
\begin{lemma}\label{l:almost}
For any $\eps>0$ there exists $\delta>0$ such that for any $\mgr\in \bK_X$ and any measure $\msp$ on~$X$, if ${\mE-\mE'}<\delta{\mE}$, then with the $\mgr$-probability at least $1-\eps$ the total variation $\|f_*(\rho)-\rho'\|_{TV}$ does not exceed $\eps \mE$.
\end{lemma}

Lemma \ref{l:almost} suffices to establish the second part of Theorem~\ref{p.max.full}. Indeed:
\begin{proof}[Proof of the second part of Theorem~\ref{p.max.full}]
Assume that for any $\delta>0$ there exists a measure $\msp$ on $X$ and $\mgr\in \bK_X$ such that
${\mE'}>(1-\delta)\mE$. Then ${\mE-\mE'}<\delta{\mE}$,
and we can apply Lemma \ref{l:almost}.

It implies that for any $\eps>0$ there exist measures $\mgr\in \bK_X$ and $\msp$ on $X$ such that for the normalized probability measure $\frac{1}{\mE}\rho$ most of its images (at least of $\mgr$-measure $1-\eps$) are within $2\eps$ from each other in the sense of total variation distance.

Passing to a weak accumulation point of $(\mgr,\frac{1}{\mE}\rho,\frac{1}{\mE'}\rho')$ as $\eps\to 0$, we can find a measure $\bar{\mgr}\in \bK_X$ and measures $\bar{\rho}$ and $\bar{\rho}'$ on $X$, such that $\bar{\mgr}$-almost surely the image of $\bar{\rho}$ is $\bar{\rho}'$. This contradicts the assumption on the set of measures~$\bK_X$.

Hence, for some $\delta_0>0$ for any measure $\msp$ on $X$ and $\mgr\in \bK_X$ we have ${\mE'}\le (1-\delta_0)\mE$. Therefore under the assumptions of the second part of Theorem~\ref{p.max.full} the sums of squares of weights of atoms of the measures
$\mgr_{n}* \dots *\mgr_{1}*\msp$
decay exponentially fast as $n$ increases. In particular, for any measure $\msp$ on $X$ and any $\mgr_1,\dots,\mgr_n$  we get
\[
\mathfrak{Max}\left(\mgr_{n}* \dots* \mgr_{1}*\msp\right) \le (1-\delta_0)^{n/2}.
\]
This completes the proof of the second part of Theorem~\ref{p.max.full}, assuming that Lemma~\ref{l:almost} holds.
\end{proof}

Let us now prove Lemma~\ref{l:almost}.
\begin{proof}[Proof of Lemma~\ref{l:almost}]

For any $\eps'>0$, let us say that an atom $y_j$ is \emph{$\eps'$-stable} if for its random preimage $x_i$ we have $\left|\frac{m_i-m_j'}{m'_j}\right|<\eps'$ with the probability at least $1-\eps'$. Let $S$ and $U$ be respectively the set
of $\eps'$-stable and unstable atoms.

For any non-$\eps'$-stable atom $y_j$ with probability at least $\eps'$ we have $|m_i-m_j'|\ge \eps' m_j'$. This implies that
$$
D_j\ge \eps'(\eps' m_j')^2=\eps'^3 \cdot m_j'^2,
$$
and hence
$$
D\ge \sum_{y_j\in U} D_j \ge \sum_{y_j\in U} \eps'^3 \cdot m_j'^2 = \eps'^3 \cdot \rho'(U). 
$$
Due to (\ref{e.zd}) we have $\mE\ge \mE'+ D$, and by assumption ${\mE-\mE'}<\delta{\mE}$, therefore
$D\le \mE-\mE'<\delta{\mE}$, and $\mE'> (1-\delta)\mE$.
 Hence we have
\begin{equation}\label{eq:E-prime}
(\eps')^3 \cdot \rho'(U) \le D \le \delta \mE < \frac{\delta}{1-\delta} \mE',
\end{equation}
thus the proportion of the unstable atoms (in terms of the $\rho'$-measure) is at most
$$
\kappa:=\frac{\delta}{(1-\delta) (\eps') ^3}.
$$
Notice that that for any fixed $\eps'$ the proportion $\kappa$ can be made arbitrarily small by a choice of sufficiently small~$\delta$.

For each $f$, denote
\begin{multline*}
   Y_f:=\{y_j, j=1, 2, \ldots \ \ \left| \ \ y_j \text{ is $\eps'$-stable}, \text{and} \right. \\ \frac{|m_i-m'_j|}{m'_j}<\eps', \ \text{where}\  m_i \ \text{is the weight of} \ x_i=f^{-1}(y_j) \}.
\end{multline*}

Definition of $\eps'$-stable atoms implies that at each of the stable atoms $y_j$ the probability that it is the image of some $x_i$ with $\frac{|m_i-m'_j|}{m'_j}<\eps'$ is at least $1-\eps'$; in other words, for any individual stable atom $y_j$ we have $\mgr(\{f : \, y_j\in Y_f\})\ge 1-\eps'$. Integrating it with respect to $\rho'$ and applying Fubini's theorem (see Fig.~\ref{f:Yf}), we get
\begin{equation}\label{eq:E-Yf}
\E_{\mgr} \rho'(Y_f) =\sum_{y\in S} \rho'(y) \cdot \mgr(\{f: y \in Y_f\}) \ge (1-\eps') \rho'(S) \ge (1-\eps') \cdot (1-\kappa) \mE'.
\end{equation}
\begin{figure}
\includegraphics{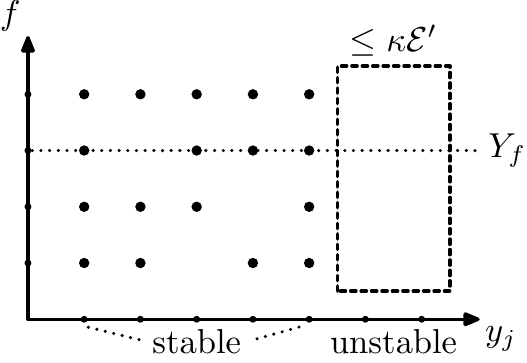}
\caption{Set $(y_j,f)$ such that $y_j$ is $\eps'$-stable, $y_j=f(x_i)$ and $\frac{|m_i-m'_j|}{m'_j}<\eps'$.}\label{f:Yf}
\end{figure}
From here we get
$$
\E_{\mgr} (\mE'-\rho'(Y_f)) \le (1-(1-\eps') (1-\kappa)) \mE'\le (\eps'+\kappa) \mE'.
$$
and hence the Markov's inequality implies that for any $\eps''>0$,
$$
\mgr(\{ f \mid \rho'(Y_f)>(1-\eps'') \mE' \}) =
\mgr(\{ f \mid \mE'- \rho'(Y_f) <\eps'' \mE'  \})
\ge
1-\frac{\eps'+\kappa}{\eps''}.
$$
Choose $\eps''=\frac{\eps}{6}$. Then choose sufficiently small $\eps'>0$ to make sure that $\frac{\eps'}{\eps''}<\frac{\eps}{2}$ and $\eps'(2+\eps') <\frac{\eps}{6}$ (the reason for the latter requirement will be clear soon, see Eq.~\eqref{eq:eps-pp} below). Finally, let us choose $\delta>0$ sufficiently small so that $\frac{\delta}{1-\delta}<\frac{\eps}{3}$ and  $\frac{\kappa}{\eps''}=\frac{\delta}{(1-\delta) (\eps') ^3\eps''}<\frac{\eps}{2}$.
 Then we have
$$
\mgr(\{ f \mid \rho'(Y_f)>(1-\eps'') \mE' \})  \ge 1-\frac{\eps'+\kappa}{\eps''} \ge 1- \frac{\eps}{2} - \frac{\eps}{2} =1-\eps.
$$
Therefore, with the $\mgr$-probability at least $1-\eps$ we have
\begin{equation}\label{e.zz}
\rho'(Y_f)>(1-\eps'') \mE'.
\end{equation}
Let us show that (\ref{e.zz}) implies that  $\|f_*(\rho)-\rho'\|_{TV}\le \eps \mE$. Indeed,
the total variation $\|f_*(\rho)-\rho'\|_{TV}$ can be bounded from above:
$$
\|f_*(\rho)-\rho'\|_{TV} \le \sum_{y\in Y_f} |\rho(f^{-1}(y))-\rho'(y)| + (\mE'-\rho'(Y_f)) + (\mE-\rho(f^{-1}(Y_f)).
$$
We have
\begin{multline*}
(\mE-\rho(f^{-1}(Y_f)) \le (\mE-\mE') + (\mE'-\rho'(Y_f)) + |\rho'(Y_f)-\rho(f^{-1}(Y_f))| \le
\\
\le (\mE-\mE') + (\mE'-\rho'(Y_f)) + \sum_{y\in Y_f} |\rho(f^{-1}(y))-\rho'(y)|,
\end{multline*}
therefore
\begin{equation}\label{eq:TV-bound}
\|f_*(\rho)-\rho'\|_{TV} \le 2\sum_{y\in Y_f} |\rho(f^{-1}(y))-\rho'(y)| + 2 (\mE'-\rho'(Y_f)) + (\mE-\mE').
\end{equation}
Recall that by assumption $\mE-\mE'<\delta \mE$;
this implies (see~\eqref{eq:E-prime})
\begin{equation}\label{e.z}
\mE-\mE'< \delta \mE< \frac{\delta}{1-\delta}\mE'.
\end{equation}
If $y_j\in Y_f$, then $\left|\frac{m_i-m_j'}{m_j'}\right|<\eps'$, and hence
\begin{equation}\label{eq:eps-pp}
|m_i^2-(m_j')^2|=\left|\frac{m_i-m_j'}{m_j'}\right|\cdot\left|\frac{2m_j'+(m_i-m_j')}{m_j'}\right|(m_j')^2\le \eps'(2+\eps')(m_j')^2.
\end{equation}
Therefore,
\begin{multline}\label{e.z1}
\sum_{y\in Y_f} |\rho(f^{-1}(y))-\rho'(y)| \le \sum_{y\in Y_f} \eps' (2+\eps') \rho'(y) = \\ =  \eps' (2+\eps')\rho'(Y_f)\le \eps'(2+\eps') \mE'.
\end{multline}
Combining (\ref{e.zz}), (\ref{eq:TV-bound}), (\ref{e.z}), and (\ref{e.z1}), and due to the choice of $\eps', \eps''$, and $\delta$ above, we get
\begin{multline*}
\|f_*(\rho)-\rho'\|_{TV} \le \left(2\eps'(2+\eps')+2\eps''+  \frac{\delta}{1-\delta}\right) \mE' <
\left(2\frac{\eps}{6}+2\frac{\eps}{6}+\frac{\eps}{3}\right) \mE' =  \eps \mE'\le \eps \mE,
\end{multline*}
which concludes the proof of Lemma \ref{l:almost}.
\end{proof}

We are now ready to complete the proof of Theorem~\ref{p.max.full}:
\begin{proof}[Proof of the first part of Theorem~\ref{p.max.full}]
Assume the contrary: there exists $\eps>0$ such that for arbitrary large $n$ there exists a probability measure $\msp$ on $X$ and measures $\mgr_1,\dots,\mgr_n\in\bK_X$, such that
\[
\mathfrak{Max}\left( \mgr_n*\dots*\mgr_1*\msp \right) \ge \eps.
\]

For a measure $\msp$, let
\[
\mE(\msp):=\sum_{\msp(\{x\})>0} \msp(\{x\})^2
\]
be the $\ell_2$-norms of the vector of weights of its atoms. Take the sequence of intermediate iterates
\[
\msp_k:=\mgr_k*\dots*\mgr_1*\msp, \quad k=0,\dots, n,
\]
and consider the corresponding sequence $\mE(\msp_k)$. These squared norms form a non-increasing (due to~\eqref{e.zd}) sequence, starting with $\mE(\msp_0)\le 1$ and ending with $\mE(\msp_n)\ge \eps^2$. Hence, for every $n$ there exists $k$ such that
\begin{equation}\label{eq:Z-n}
\mE(\msp_{k+1})\ge \eps^{2/n} \mE(\msp_k).
\end{equation}
As $\eps>0$ is fixed, and $n$ can be chosen arbitrarily large,~\eqref{eq:Z-n}
implies that for every $\delta>0$ there exist measures
\[
\msp'=\msp_k, \, \msp''=\msp_{k+1}=\mgr_{k+1}*\msp',
\]
such that $\mE(\msp'')>(1-\delta) \mE(\msp')$.

Let us now apply the argument of the proof of Lemma~\ref{l:almost}: consider the normalized measures
\begin{equation}\label{eq:sq-atoms}
\frac{1}{\mE(\msp')} \sum_{\msp'(x_i)>0} (\msp'(x_i) ^2 \Dirac_{x_i}), \quad
\frac{1}{\mE(\msp'')} \sum_{\msp''(x_i)>0} (\msp'(x_i) ^2 \Dirac_{x_i}).
\end{equation}

The same arguments as above imply that any accumulation point of these measures and of the corresponding $\mgr$'s  is a measure $\bar{\msp}$ with a deterministic image and its image $\bar{\msp}'$:
\begin{equation}\label{eq:f-mu-prime}
f_* \bar{\msp}=\bar{\msp}' \quad \text{for $\mgr$-a.e. } \, f\in \Homeo(X)
\end{equation}

In itself, it wouldn't be a contradiction, as we no longer assume the absence of a measure with a deterministic image. However, the maximal weight of an atom cannot be increased by a convolution, and hence before passing to the limit one has always
\[
\mathfrak{Max}(\msp'),\mathfrak{Max}(\msp'') \ge \mathfrak{Max}\left( \mgr_n*\dots*\mgr_1*\msp \right) \ge \eps.
\]
Hence, the maximal weights of the normalized measures~\eqref{eq:sq-atoms} are at least $\eps^2$, and thus (passing to the limit) one has
\[
\mathfrak{Max}(\bar{\msp}),\mathfrak{Max}(\bar{\msp}')\ge \eps^2.
\]
Finally, consider the set of atoms of maximal weight (that is the same for $\bar{\msp}$ and for~$\bar{\msp}'$); denote
\[
\eps':=\mathfrak{Max}(\bar{\msp})=\mathfrak{Max}(\bar{\msp})
\]
and let
\[
F:= \{x: \bar{\msp}(x) =\eps' \}, \quad F':= \{x: \bar{\msp}(x) = \eps'\}.
\]
Then these are two finite sets (consisting of at most $\frac{1}{\eps'}$ points), and~\eqref{eq:f-mu-prime} implies that
\begin{equation}\label{eq:f-mu-prime-new}
f_* F=F' \quad \text{for $\mgr$-a.e. } \, f\in \Homeo(X)
\end{equation}

Hence, $\bar{X}$ is a finite set with a deterministic image, and this provides us with the desired contradiction.
\end{proof}

\subsection{Linear case: a point avoids a subspace}\label{s:planes}

This section is devoted to the proof of Proposition~\ref{p:proj-dissolving}

Before passing to its proof, let us establish an ancillary lemma:
\begin{lemma}\label{l:L-M}
Let $2\le m\le d-1$, and assume that $\msp$ is a measure on $\RP^{d-1}$, such that for any $L\in \Gr(m-1,d)$ one has $\msp([L])\le \frac{\eps}{4}$. Then there exists at most $N=\left\lfloor\frac{2^{2m-1}}{\eps^{m}}\right\rfloor$ subspaces $L'\in \Gr(m,d)$ such that $\msp([L'])\ge \frac{\eps}{2}$.
\end{lemma}
\begin{proof}[Proof of Lemma~\ref{l:L-M}]
Consider the random variable $\Xi$, taking values in $\Gr(m,d) \cup \{\mathrm{NDef}\}$, defined in the following way:
\begin{itemize}
\item Take $m$ points $y_1,\dots, y_m\in \RP^{d-1}$ to be randomly and independently chosen with respect to~$\msp$.
\item If the lines in $\R^d$ corresponding to these points are contained in a subspace of dimension at most $(m-1)$, we set $\Xi=\mathrm{NDef}$.
\item Otherwise, let $\Xi$ be the unique $m$-dimensional subspace passing through these lines; in other words, $[\Xi]\subset \RP^{d-1}$ is the unique $(m-1)$-dimensional projective subspace passing through $y_1,\dots,y_m$.
\end{itemize}
Now, note that for any $L'\in \Gr(m,d)$ such that $\msp([L'])\ge \frac{\eps}{2}$ one has
\begin{equation}\label{eq:L'}
\Prob (\Xi=L') \ge \frac{\eps^{m}}{2^{2m-1}}.
\end{equation}
Indeed, one has $\Prob(y_1\in [L']) =\msp([L']) \ge \frac{\eps}{2}$. Next, if the points $y_1,\dots, y_{j-1}$ are in $[L']$ and in general position, chose a $(m-2)$-dimensional projective subspace $[L]$ passing through them. Then the conditional probability that the next point $y_j$ belongs to $[L']$, but does not belong to the $(j-2)$-dimensional projective subspace passing through $y_1,\dots,y_{j-1}$, is at least
\[
\msp([L']\setminus [L]) = \msp([L']) - \msp([L]) \ge \frac{\eps}{2}-\frac{\eps}{4}=\frac{\eps}{4}
\]
due to the assumptions of the lemma.

Multiplying such conditional probabilities for $j=2,\dots,m$, we get a lower bound for $\Prob (\Xi=L')$ by
\[
\Prob (\Xi=L') \ge \frac{\eps}{2} \cdot \left(\frac{\eps}{4}\right)^{m-1} = \frac{\eps^m}{2^{2m-1}},
\]
thus obtaining the desired~\eqref{eq:L'}. Finally, as we have this lower bound for the probability of the event $\Xi=L'$ for any $L'$ satisfying $\msp([L'])\ge \frac{\eps}{2}$, the number of such subspaces cannot exceed the inverse of this lower bound, and thus its integer part $N=\left\lfloor\frac{2^{2m-1}}{\eps^{m}}\right\rfloor$.
\end{proof}

\begin{proof}[Proof of Proposition~\ref{p:proj-dissolving}]
The proof is by induction on $m$. The base $m=1$ coincides with the Atoms Dissolving Lemma, as $(m-1)$-dimensional projective plane~$[L]$ is then a point in~$\RP^{d-1}$.

For the induction step, we will use Lemma~\ref{l:L-M}. To do so, assume that the statement is already established for some $m-1<d-1$, and let us establish it for~$m$.
Let $\eps>0$ be chosen; take $k':=k_{m-1}(\frac{\eps}{4})$. Take any $k'$ steps $\mgr_1,\dots, \mgr _{k'} \in \bK$ and denote
\[
\msp:=\mgr_1*\dots*\mgr_{k'}* \Dirac_{x_0}.
\]
Then, by choice of $k'$, for any $(m-1)$-dimensional subspace $L\subset \R^d$ one has
\[
\msp([L])=\Prob_{\mgr_1,\dots,\mgr_{k'}} ((f_{k'}\circ\dots\circ f_1)(x_0)\in [L])\le \frac{\eps}{4}.
\]
By Lemma~\ref{l:L-M}, this implies that there exists at most $N=\left\lfloor\frac{2^{2m-1}}{\eps^{m}}\right\rfloor$ subspaces $L'\in \Gr(m,d)$ such that $\msp([L'])\ge \frac{\eps}{2}$.

Now, consider the induced action on the space of $m$-dimensional subspaces $\Gr(m,d)$.  Note that due to the ``spaces condition'' in the assumptions, there are no finite sets of subspaces with deterministic image. Hence, the first part of Atoms Dissolving Theorem~\ref{p.max.full} is applicable. Thus, there exists $k''$ such that for any steps $\mgr_{k'+1},\dots, \mgr _{k'+k''}\in \bK$ and any two $m$-dimensional subspaces $L', L''$ one has
\[
\Prob_{\mgr_{k'+1}*\dots*\mgr_{k'+k''}} ((f_{k'+k''}\circ\dots\circ f_{k'+1})([L'])= [L'']) < \frac{\eps}{2N},
\]
where again $f_i$ are distributed w.r.t.~$\mgr_i$.

Now, let us show that we can take $k_m(\eps):=k'+k''$. Indeed, take any $\mgr_1,\dots,\mgr_{k'+k''}\in \bK$, any $x_0\in \RP^{d-1}$ and any $L''\in \Gr(m,d)$. As before, let
\[
\msp:=\mgr_1*\dots*\mgr_{k'}* \Dirac_{x_0},
\]
and let us decompose
\[
f_{k'+k''}\circ \dots \circ f_1 =  \underbrace{(f_{k'+k''}\circ \dots \circ f_{k'+1})}_{G} \circ \underbrace{(f_{k'}\circ \dots \circ f_1)}_{F} = G\circ F.
\]
Take all the $m$-dimensional subspaces $L'\subset \R^d$ satisfying $\msp(L')\ge \frac{\eps}{2}$; let
\[
L_1,\dots,L_{N'}\in \Gr(m,d), \quad N'\le N
\]
be the full list of such subspaces.

Inverting the last $k''$ of $k'+k''$ applied random maps, for any $L''\in \Gr(m,d)$ one gets
\begin{multline}\label{eq:G-final}
\Prob_{\mgr_1,\dots,\mgr_{k'+k''}} ((f_{k'+k''}\circ \dots \circ f_1)(x_0)\in [L'']) = \\
\Prob_{(F,G)\sim (\mgr_{k'}*\dots*\mgr_1,\mgr_{k'+k''}*\dots*\mgr_{k'+1})} \left(F(x_0)\in G^{-1}([L''])\right) = \\
= \E_{G\sim \mgr_{k'+k''}*\dots*\mgr_{k'+1}} \msp\left(G^{-1} ([L''])\right).
\end{multline}

\begin{figure}[!h!]
\includegraphics[width=0.7\textwidth]{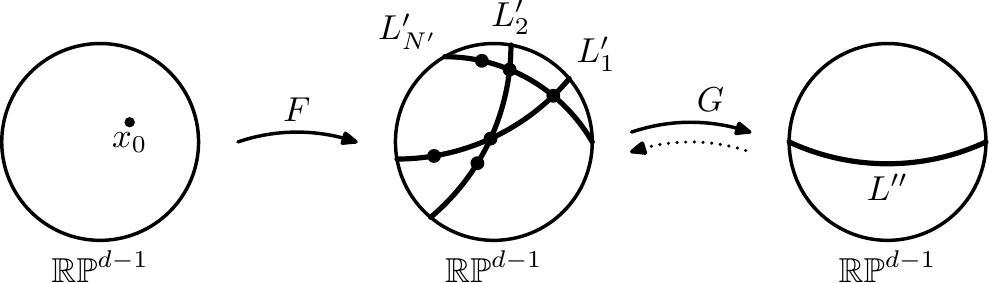}
\caption{Initial point, $N'$ intermediate subspaces of $\msp$-measure at least $\frac{\eps}{2}$, and the subspace~$L''$}\label{f:decomposition}
\end{figure}

Let us decompose (see Fig.~\ref{f:decomposition}) the expectation in the right hand side of~\eqref{eq:G-final} into two parts, depending on whether the preimage~$G^{-1} ([L''])$ is of $\msp$-measure greater or less than~$\frac{\eps}{2}$:
\begin{multline}\label{eq:G-bound}
\E_{G\sim (\mgr_{k'+k''}*\dots*\mgr_{k'+1})} \, \msp\left(G^{-1} ([L''])\right)=
\\ =
\E_{G\sim (\mgr_{k'+k''}*\dots*\mgr_{k'+1})}
\left[\msp\left(G^{-1} ([L''])\right)   \Ind_{\msp\left(G^{-1} ([L''])\right)< \frac{\eps}{2}}\right] \\
+
\E_{G\sim (\mgr_{k'+k''}*\dots*\mgr_{k'+1})}
\left[\msp\left(G^{-1} ([L''])\right)   \Ind_{\msp\left(G^{-1} ([L''])\right)\ge \frac{\eps}{2}}\right] \\
\end{multline}
Now, the former summand in the right hand side is strictly smaller than~$\frac{\eps}{2}$, as it is a strict upper bound for $\msp\left(G^{-1} ([L''])\right)$ whenever the indicator doesn't vanish. On the other hand, for the latter summand it suffices to use  $\msp\left(G^{-1} ([L''])\right)\le 1$ to obtain
\begin{multline}
\E_{G\sim (\mgr_{k'+k''}*\dots*\mgr_{k'+1})}
\left[\msp\left(G^{-1} ([L''])\right)   \Ind_{\msp\left(G^{-1} ([L''])\right)\ge \frac{\eps}{2}}\right] \le
\\
\le \Prob_{G\sim (\mgr_{k'+k''}*\dots*\mgr_{k'+1})}   \left(\msp\left(G^{-1} ([L''])\right) \ge \frac{\eps}{2} \right)=
\\
= \sum_{j=1}^{N'}
\Prob_{G\sim (\mgr_{k'+k''}*\dots*\mgr_{k'+1})}  \left( G^{-1} ([L'']) = [L_j]  \right) \le
\\
\le N'\cdot \frac{\eps}{2N} \le \frac{\eps}{2}.
\end{multline}
Adding these two estimates, we obtain
\begin{multline*}
\Prob_{\mgr_1,\dots,\mgr_{k'+k''}} ((f_{k'+k''}\circ \dots \circ f_1)(x_0)\in [L'']) = \\
= \E_{G\sim \mgr_{k'+k''}*\dots*\mgr_{k'+1}} \msp\left(G^{-1} ([L''])\right)
<\frac{\eps}{2} + \frac{\eps}{2} = \eps.
\end{multline*}
This completes the induction step; indeed, for any $k\ge k_m(\eps)$ it suffices to consider the last applied $k_m(\eps)$ maps and average over the possible images of $x_0$ during the initial $(k-k_m(\eps))$ ones.
\end{proof}

\begin{appendix}
\section{}\label{a.1}

In the stationary case, assuming the absence of an invariant measure for the action on $\RP^{d-1}$, together with the moments condition, suffices to ensure the existence and the positivity of the Lyapunov exponent; see~\cite{Vi}. Here we present an example, showing that for the nonstationary setting this is no longer the case, that is, that the assumptions of Theorem~\ref{t.1} with the finite moment condition but without the spaces condition do not suffice to obtain the conclusion of Theorem~\ref{t.L}.

Namely, here we present an example of a sequence of probability measures $\mgr_n$ on $\SL(4, \mathbb{R})$, for which the assumptions of Theorem~\ref{t.1} hold, the norms $\|A_j\|$ are uniformly bounded, but there is no sequence $\{\mathcal{L}_n\}$ of real numbers such that almost surely $\lim_{n\to \infty}\frac{1}{n}\left(\log \|T_n\|-\mathcal{L}_n\right)=0$.

Let $B$ be a random $\SL(2, \mathbb{R})$ matrix given in the following way. With probability $1/2$ we set $B=\begin{pmatrix}
2 & 0 \\
0 & 1/2 \\
\end{pmatrix}$,
and with probability $1/2$ we choose $B$ as a rotation by a random angle uniformly distributed on the circle.

By Furstenberg Theorem, there exists $\lambda>0$ such that almost surely there exists
\begin{equation}\label{eq:B-lambda}
\lim_{n\to \infty}\frac{1}{n}\|B_n\ldots B_1\|=\lambda>0,
\end{equation}
where $\{B_i\}$ are i.i.d., chosen with respect to the same law as described above.

In our example we will only have two distributions on $\SL(4, \mathbb{R})$, we will denote them by $\mgr_\alpha$ and $\mgr_\beta$. The distribution $\mgr_\alpha$ will be given by block-diagonal matrices
\begin{equation}\label{e.B}
\hB=\begin{pmatrix}
  \frac{1}{100}B^{(1)} & 0 \\
  0 & 100B^{(2)} \\
\end{pmatrix},
\end{equation}
where $B^{(1)}$ and $B^{(2)}$ are independent random matrices, distributed with respect to the same law as the matrix $B$ above. The distribution $\mgr_\beta$ will be defined in the following way. Let
\[
M:=\begin{pmatrix}
  0 & 0 & 1 & 0 \\
  0 & 0 & 0 & 1 \\
  1 & 0 & 0 & 0 \\
  0 & 1 & 0 & 0 \\
\end{pmatrix}
\]
be the matrix that interchanges two copies of $\R^2$. Then, we choose the matrix $\hB$ as in~\eqref{e.B}, and take the final random matrix $A:=Q\hB$, where
$$
Q= \begin{cases}
\Id & \text{ with probability $1/2$},\\
M& \text{ with probability $1/2$}
\end{cases}
$$
is chosen independently from~$B$. It is not hard to see that the assumptions of Theorem~\ref{t.1} are satisfied for both distributions $\mgr_\alpha$ and $\mgr_\beta$. Then, for an appropriate choice of the sequence $\mgr_i$, at some moments $n$ the law of $\frac{1}{n} \log \|T_n\|$ is bimodal, preventing it from converging to a deterministic sequence. Namely, we have the following statement.

\begin{example}\label{ex:R-4}
Choose an increasing subsequence $\{n_m\}_{m\in \mathbb{N}}$ of natural numbers, such that
\[
\lim_{m\to\infty} \frac{n_{m+1}}{n_m} = \infty;
\]
for instance, one can take $n_m = 10^{m!}$. Let
$$
\mgr_n=\left\{
        \begin{array}{ll}
          \mgr_\beta, & \hbox{if $n\in\{n_m \mid m\in \mathbb{N}\}$;} \\
          \mgr_\alpha, & \hbox{otherwise.}
        \end{array}
      \right.
$$
Take the matrices $A_i\in \SL(4, \mathbb{R})$ be chosen independently with respect to the distribution $\mgr_i$, and set $T_n=A_n\cdot \ldots \cdot A_1$. Take $n_m':=2n_m$, and define
\[
\eta_m:= \begin{cases}
\lambda + \log 100 & \text{ if } Q_{m}=\Id,\\
\lambda & \text{ if } Q_{m}=M,
\end{cases}
\]
where $A_{n_m}=Q_{m} \hB_{n_m}$. Then, almost surely we have
\begin{equation}\label{eq:n-m-prime}
\frac{1}{n_m'} \log \|T_{n_m'}\|-\eta_m \xrightarrow[m\to\infty]{} 0.
\end{equation}
In particular, as all the random variables $\eta_m$ are independent and take values $\lambda$ and $\lambda+\log 100$ with probability $1/2$,
there is no sequence $\{\mathcal{L}_n\}$ of real numbers such that almost surely
\[
\lim_{n\to \infty}\frac{1}{n}\left(\log \|T_n\|-\mathcal{L}_n\right)=0,
\]
as such a statement would fail even on the subsequence~$n'_m$.

Moreover, if instead one chooses $n_m'$ in a way that $\frac{n_m'}{n_m} \to 1+c$, where $c>0$, then~\eqref{eq:n-m-prime} still holds with the random variable $\eta_{m,c}$ taken instead of $\eta_m$, where
\[
\eta_{m,c} = \begin{cases}
\lambda + \log 100 &  \text{ if } Q_{m}=\Id,\\
\lambda + \frac{|1-c|}{1+c}\log 100&  \text{ if } Q_{m}=M.
\end{cases}
\]
\end{example}
\begin{proof}
Note first that the norm of each $A_i$ is equal to $2\cdot 100=200$, and hence the contribution of the product $T_{n_{m-1}}$ of the first $n_{m-1}$ matrices to $\frac{1}{n'_m} \log \|T_{n'_m}\|$ does not exceed
\[
\frac{n_{m-1}}{n'_m} \log(200) \to 0, \quad m\to \infty.
\]
Thus, if we replace
\[
T_{n'_m} = \hB_{n'_m} \dots \hB_{n_m +1} Q_m \hB_{n_m} \dots \hB_{n_{m-1} +1} Q_{m-1} \hB_{n_m} \dots \hB_{n_1+1}Q_1 \hB_{n_1} \dots \hB_{1}
\]
with the product from which all $Q_1,\dots,Q_{m-1}$ are removed,
\[
\wT_{m} := \hB_{n'_m} \dots \hB_{n_m +1} Q_m \hB_{n_m} \dots \hB_{1},
\]
then
\[
\frac{1}{n'_m} \left| \log \|\wT_m\| - \log \|T_{n'_m}\|\right| \le \frac{n_{m-1}}{n'_m}\cdot 2 \log(200) \to 0,
\]
and hence it suffices to establish~\eqref{eq:n-m-prime} with $\wT_m$ instead of~$T_{n'_m}$.

Now, the products before and after $Q_m$ in $\wT_m$ are respectively equal to
\[
\begin{pmatrix}
  \frac{1}{100^{n_m}} B_{n_m}^{(1)}\dots B_{1}^{(1)} & 0 \\
  0 & 100^{n_m} B_{n_m}^{(2)}\dots B_{1}^{(2)} \\
\end{pmatrix}
\]
and
\[
\begin{pmatrix}
  \frac{1}{100^{(n'_m-n_m)}} B_{n'_m}^{(1)}\dots B_{n_m+1}^{(1)} & 0 \\
  0 & 100^{(n'_m-n_m)} B_{n'_m}^{(2)}\dots B_{n_m+1}^{(2)} \\
\end{pmatrix}.
\]
Thus, the full product $\wT_m$ is a block matrix: it has two zero $2\times 2$ blocks and two blocks that are products of $n'_m$ independent matrices $B^{(i)}_j$ and a scalar $100^{\pm k_m},$ where
\[
k_m =  \begin{cases}
n'_m &  \text{ if } Q_{m}=\Id,\\
n_m-(n'_m-n_m)&  \text{ if } Q_{m}=M.
\end{cases}
\]

Using~\eqref{eq:B-lambda}, one gets the almost sure representation for the logarithm of the norm $\log \|\wT_m\|$ as a sum of
\[
\begin{cases}
n'_m \log 100 &  \text{ if } Q_{m}=\Id,\\
|n'_m-2n_m| \log 100&  \text{ if } Q_{m}=M,
\end{cases}
\]
corresponding to the powers of 100, and of $n'_m (\lambda + o(1))$, corresponding to the norms of products of $B^{(i)}_j$; dividing by $n'_m$ and passing to the limit provides the desired~\eqref{eq:n-m-prime}.
\end{proof}

\end{appendix}

\section*{Acknowledgments}

We are grateful to David Damanik, Ilya Goldsheid, and Lana Jitomirskaya
for helpful discussions and useful remarks.

\end{document}